\documentclass[a4paper,11pt]{amsart}
\usepackage[utf8]{inputenc}
\usepackage[T1]{fontenc}
\usepackage[UKenglish]{babel}
\usepackage[margin=20mm]{geometry}
\usepackage{color}

\usepackage{graphicx}

\usepackage{amsmath,amssymb,amsfonts,amsthm}
\usepackage{mathrsfs,eucal,dsfont}
\usepackage{verbatim,enumitem}

\usepackage{hyperref,url}
\usepackage{amssymb, amsmath, amsthm, amscd}

\usepackage{eucal,mathrsfs,dsfont}
\usepackage{color}

\newcommand{\R}{\mathds R}
\newcommand{\Z}{\mathds Z}

\newcommand{\Pp}{\mathds P}
\newcommand{\Ee}{\mathds E}

\newcommand{\I}{\mathds 1}

\def\<{\langle}
\def\>{\rangle}
\newcommand{\D}{\textup{div}}
\newcommand{\leb}{\textup{Leb}}

\newtheorem{theorem}{Theorem}[section]
\newtheorem{lemma}[theorem]{Lemma}
\newtheorem{proposition}[theorem]{Proposition}
\newtheorem{corollary}[theorem]{Corollary}
\numberwithin{equation}{section}

\theoremstyle{definition}

\newtheorem{remark}[theorem]{Remark}

\begin{document}
\allowdisplaybreaks
\title[Approximation of heavy-tailed distributions] {\bfseries
Approximation of heavy-tailed distributions via stable-driven SDEs}

\author{Lu-Jing Huang\qquad Mateusz B. Majka \qquad  Jian Wang}

\thanks{\emph{L-J Huang:}
 College of Mathematics and Informatics, Fujian Normal University, 350007 Fuzhou, P.R. China.
   \texttt{lujingh@yeah.net}}

\thanks{\emph{M.\ B.\ Majka:}
School of Mathematical and Computer Sciences, Heriot-Watt University, Edinburgh, EH14 4AS, UK. \texttt{m.majka@hw.ac.uk}}

\thanks{\emph{J. Wang:}
 College of Mathematics and Informatics \& Fujian Key Laboratory of Mathematical
Analysis and Applications (FJKLMAA) \& Center for Applied Mathematics of Fujian Province (FJNU), Fujian Normal University, 350007 Fuzhou, P.R. China.
   \texttt{jianwang@fjnu.edu.cn}}

\date{}
\maketitle

\begin{abstract}

Constructions of numerous approximate sampling algorithms are based on the well-known fact that certain Gibbs measures are stationary distributions of ergodic stochastic differential equations (SDEs) driven by the Brownian motion. However, for some heavy-tailed distributions it can be shown that the associated SDE is not exponentially ergodic and that related sampling algorithms may perform poorly. A natural idea that has recently been explored in the machine learning literature in this context is to make use of stochastic processes with heavy tails instead of the Brownian motion. In this paper we provide a rigorous theoretical framework for studying the problem of approximating heavy-tailed distributions via ergodic SDEs driven by symmetric (rotationally invariant) $\alpha$-stable processes.


\medskip

\end{abstract}
\allowdisplaybreaks

\section{Introduction}\label{section1}
Suppose we are given a probability distribution $\mu$ on $\R^d$ defined via
\begin{equation}\label{Gibbs}
\mu(dx) = Z^{-1} \exp \left( -V(x) \right) dx \,,
\end{equation}
where $V: \R^d \to \R$ is the potential, and $Z:= \int_{\R^d} \exp \left( -V(x) \right) dx$ is the normalizing constant. The goal in approximate sampling is to generate a sequence of probability measures $(\mu_k)_{k\ge1}$ such that for sufficiently large $k$ the measure $\mu_k$ constitutes a good approximation of $\mu$. This can be achieved e.g.\ by
utilizing a
stochastic process
with the unique stationary distribution $\mu$.
If we can show that this process is exponentially ergodic, then
we can use it to construct an algorithm for approximate sampling from $\mu$ that,
under some assumptions on $V$ in \eqref{Gibbs},
converges exponentially fast regardless of its initial condition.

A commonly used example of such a process is the solution $(X_t)_{t \ge 0}$ to the (overdamped) Langevin SDE
\begin{equation}\label{LangevinSDE}
dX_t = -
\nabla V(X_t)\, dt +
\sqrt{2}\,dB_t \,,
\end{equation}
where $(B_t)_{t \ge 0}$ is the standard Brownian motion in $\R^d$. If the potential $V$ is sufficiently regular,
it can be easily shown that $\mu$ given by \eqref{Gibbs} is a stationary distribution of $(X_t)_{t \ge 0}$. Moreover, there are many results on the exponential ergodicity of \eqref{LangevinSDE} under relatively weak dissipativity conditions on $V$, see e.g.\ \cite{Kha} and the references therein for approaches based on Lyapunov-type drift conditions, the monographs \cite{BGL, MF, WangBook} for methods based on functional inequalities, and \cite{MF, WangBook} for probabilistic coupling techniques (in particular, \cite{Eberle2016, EberleGuillinZimmer2019} for a recent study on this topic).

There are numerous sampling algorithms in the literature that are based on Euler discretizations of \eqref{LangevinSDE}, cf.\ \cite{EberleMajka2019, MajkaMijatovicSzpruch2018} and the references therein. The analysis of their performance is often carried out by bounding the discretization error between the Euler scheme and the SDE, and then by directly employing ergodicity results for SDEs, see e.g.\ \cite{ChengJordan2018, Dal, DurmusMoulines2017, MFWB}. Hence the analysis of convergence
of the SDE is an important first step towards evaluating performance of such algorithms, and one usually
 cannot expect fast convergence of the algorithm without fast convergence of the associated SDE, see \cite{RobertsTweedie1996}
(with some possible exceptions discussed in \cite{Erdogdu2020}).

However, in \cite{RobertsTweedie1996} (see Theorem 2.4 and Section 2.3 therein) it has been shown that the solution to \eqref{LangevinSDE} may not be exponentially ergodic if the distribution $\mu$ defined in \eqref{Gibbs} is heavy-tailed.
Indeed, it is known that the Langevin SDE \eqref{LangevinSDE} has the generator
$Lf:=\Delta f -\nabla V \cdot\nabla f$
 which is a symmetric operator on $L^2(\R^d; \mu)$, and that the Poincar\'e inequality for $L$ (which is equivalent to the exponential ergodicity of the SDE \eqref{LangevinSDE}) implies exponential tails of $\mu$; see \cite[Theorems 1.1.1 and 1.2.5]{WangBook}. However, for heavy-tailed $\mu$, one can only expect weak-Poincar\'e inequalities, which indicates that the solution to \eqref{LangevinSDE} only converges with a polynomial
or a subexponential rate;
see \cite[Chapter 4]{WangBook} for more details.
A very natural question to ask in this context is whether instead of \eqref{LangevinSDE} one could use SDEs driven by other stochastic processes, with tails better suited for the task of approximating heavy-tailed $\mu$.

The first steps
in that direction have been taken in \cite{Simsekli2017, NguyenSimsekli2019}
(see also \cite{SimsekliZhuTeh, YeZhu} for further extensions).
The idea there is based on the fact that $\mu$ given by \eqref{Gibbs} can be shown to be a stationary distribution of
\begin{equation}\label{stableSDE}
dX_t = b(X_t)\,dt + dZ_t \,,
\end{equation}
where $(Z_t)_{t \ge 0}$ is the symmetric (rotationally invariant) $\alpha$-stable process in $\R^d$ with $d\ge1$ and $\alpha\in (1,2)$, and the drift $b(x)$ is given by
\begin{equation}\label{specialDrift}
b(x) = - {C_{d,2-\alpha}}e^{V(x)} \int_{\R^d} \frac{e^{-V(y)} \nabla V(y)}{|x-y|^{d-(2-\alpha)}}\, dy \,,
\end{equation} where
the potential
$V\in C^1(\R^d)$ is such that $e^{-V}|\nabla V|\in L^1(\R^d;dx) \cap C_b(\R^d)$,
 and
$C_{d,\alpha} := \Gamma((d-\alpha)/2)/(2^{\alpha} \pi^{d/2} \Gamma(\alpha/2) )$.
Hence, if the SDE \eqref{stableSDE} is exponentially ergodic, one could use an algorithm based on its discretization to obtain a new alternative way of approximating $\mu$ (possibly faster than algorithms based on \eqref{LangevinSDE} if $\mu$ is heavy-tailed).
The authors of \cite{Simsekli2017, NguyenSimsekli2019} called their approach Fractional Langevin Monte Carlo due to a possible interpretation of the drift \eqref{specialDrift} in terms of the Riesz potential, which is an inverse operator to the fractional Laplacian, see e.g., \cite[Section 2.7]{Kwa} and the references therein.

There are, however, several challenges to this approach, related both to verifying theoretical properties of the SDE \eqref{stableSDE} and to finding its appropriate discrete-time counterpart for use in simulations. In the present paper we focus on the former, in response to some questions that were left unanswered in \cite{Simsekli2017, NguyenSimsekli2019}. Indeed, the exponential ergodicity of \eqref{stableSDE}
has been checked in \cite{Simsekli2017, NguyenSimsekli2019} only under some very special and difficult to verify assumptions.
As we will see in Section \ref{sectionProperties}, the drift $b(x)$ defined by \eqref{specialDrift} seems to be in general only locally $(2-\alpha)$-H\"{o}lder continuous, while in the setting of \cite{Simsekli2017, NguyenSimsekli2019} it is assumed to be Lipschitz continuous
and differentiable.
Moreover, the authors of \cite{NguyenSimsekli2019} assume that $b(x)$ satisfies a contractivity at infinity condition $\langle b(x) - b(y) , x - y \rangle \le - K|x-y|^2$ for all $x$, $y \in \R^d$ such that $|x - y| > R$, with some constants $K$, $R > 0$ (cf.\ \cite[Assumption (H5) and Proposition 1]{NguyenSimsekli2019}), which also seems to be unverifiable in the general case. The lack of all these properties of $b(x)$ makes it impossible to prove the exponential ergodicity of \eqref{stableSDE} by utilizing results from the existing literature (see e.g.\ \cite{LiangMajkaWang2019} for some recent developments in this topic).
Furthermore, because of the unusual form of \eqref{specialDrift}, it is not even immediately clear whether \eqref{stableSDE} has a unique, non-explosive strong solution, which also has not been verified in \cite{Simsekli2017, NguyenSimsekli2019}.
Finally, due to non-differentiability of $b(x)$, the proof that $\mu$ given by \eqref{Gibbs} is the unique invariant probability measure for \eqref{stableSDE} cannot be as straightforward as  in \cite[Theorem 1.1]{Simsekli2017} or \cite[Theorem 1.1]{YeZhu}.
In the present paper we fill
all these gaps
by carefully deriving appropriate bounds on \eqref{specialDrift}, and by proving
all the properties of \eqref{stableSDE} mentioned above
in a rigorous way. In particular, we study the drift term $b(x)$ defined by \eqref{specialDrift} for all $d>2-\alpha$
(not only for the case of $d\ge1$ and $\alpha\in (1,2)$),
 and we  define a new drift term to treat the case of $d\le 2-\alpha$.
To this end, we will use the notion of the fractional Laplace operator (see e.g.\ \cite{BBKRSV, BV, Kwa} and the references therein), which is defined for all
$f \in C_b^2(\R^d)$ by
$$-(-\Delta)^{\alpha/2} f(x) := c_{d,\alpha} \lim_{\varepsilon \to
0}  \int_{\{ |y-x| > \varepsilon \}} \frac{f(y)-f(x)}{|y-x|^{d+\alpha}} dy,$$
where $c_{d,\alpha} := 2^{\alpha}\Gamma((d+\alpha)/2)/( \pi^{d/2}|\Gamma(-\alpha/2)|)= \alpha2^{\alpha-1}\Gamma((d+\alpha)/2)/( \pi^{d/2} \Gamma(1-\alpha/2) )$.
See e.g.\ \cite[formulas (1.3) and (1.35)]{BBKRSV} or \cite[Definition 2.5]{Kwa},
and note that $c_{d,\alpha} =
|C_{d,-\alpha}|$.
Then, in order to cover the case of $d\le 2-\alpha$, i.e., $d=1$ and $\alpha\in (0,1]$, we will work with the drift
\begin{equation}\label{e:one}
b(x)=-e^{V(x)}\int_{-\infty}^x (-\Delta)^{\alpha/2} e^{-V(u)}\,du,\quad x\in \R.
\end{equation}
Everywhere in this paper, we will be concerned with
the SDE \eqref{stableSDE} driven by a symmetric $\alpha$-stable process $(Z_t)_{t\ge0}$ on $\R^d$ with $\alpha\in (0,2)$, where the drift term $b(x)$ is defined by \eqref{specialDrift} when $d>2-\alpha$, and by \eqref{e:one} when $d\le 2-\alpha$.
We will refer to $b(x)$ as the fractional drift in both cases.
We will comment on some possible approaches to the problem of discretization of \eqref{stableSDE} in Remark \ref{remarkDiscrete}. However, our focus in this paper is the analysis of the SDE \eqref{stableSDE}, and we leave a more detailed discussion of discrete-time algorithms for future work.

For our main result, we require that
the following assumption on the potential $V$ is satisfied.

{\noindent{\bf Assumption (A)}\it\,\,
$V$ is a radial function on $\R^d$ $($and hence, by a slight abuse of notation, we write $V(x) = V(|x|)$ for all $x \in \R^d$$)$  such that
\begin{equation}\label{e:1.1}\limsup_{r\to\infty}[ e^{-V(r)} r^{d+\alpha}]<\infty,\end{equation} and
one of the following two conditions is satisfied:
\begin{itemize}
\item[{\rm(i)}] when $d>2-\alpha$,  $V\in C^1(\R^d)$, $e^{-V}|\nabla V|\in L^1(\R^d;dx)\cap C_b(\R^d)$,
        \begin{equation}\label{e:logconc}
		r_0 := \sup \{ r > 0 : V'(r) \le 0 \}<\infty,
		\end{equation}
			 and \begin{equation}\label{e:1.2}\int_{0}^\infty e^{-V(r)} |V'(r)| r^d\,dr<\infty,\quad \int_0^\infty e^{-V(r)} V'(r) r^d\,dr >0.\end{equation}
\item[{\rm(ii)}] when $d\le 2-\alpha$,
 $V\in C^2(\R)$, $e^{-V}\in L^1(\R;dx)\cap C_b^2(\R)$,
$$
 \limsup_{x\rightarrow \infty}[x^3e^{-V(x)}|V'(x)^2-V''(x)|]<\infty,
$$
 and
$$  \liminf_{x\rightarrow \infty}[x^3e^{-V(x)}(V'(x)^2-V''(x))]\ge0.$$
\end{itemize} }

We have the following result.

\begin{theorem}\label{mainResult}
	Under Assumption {\rm(A)}, the {\rm SDE} \eqref{stableSDE}
with the fractional drift $b(x)$
given by \eqref{specialDrift} when $d>2-\alpha$, and by \eqref{e:one} when $d\le 2-\alpha$,
has a unique non-explosive strong solution $X:=(X_t)_{t\ge0}$ such that the process $X$ is exponentially
 ergodic with the unique invariant probability measure $\mu$ given by \eqref{Gibbs}.
 More explicitly, for any $\beta\in [0,\alpha)$, there is a constant $\lambda>0$ such that for any $X_0 \sim \mu_0$ with finite $\beta$-moment and any $t>0$,
$$\|\mathcal{L}(X_t)-\mu\|_{{\rm Var},V_0}:=\sup_{|f|\le V_0} \left|\int_{\R^d}\Ee^xf(X_t)\,\mu_0(dx) -\mu(f)\right|\le C(\mu_0) e^{-\lambda t},$$
where $V_0(x)=(1+|x|)^{\beta}$, $C(\mu_0)$ is a positive constant,
 and $\mathcal{L}(X_t) $ denotes the distribution of $X_t$ for every $t>0$.
\end{theorem}

Note that the  weighted total variation distance $\| \cdot \|_{\operatorname{Var},V_0}$ from Theorem \ref{mainResult} dominates both the standard total variation and the $L^\beta$-Wasserstein distance (see e.g. \cite[Remark 2.3]{EberleGuillinZimmer2019}). Therefore we have the following immediate corollary.

\begin{corollary}
	 Under Assumption {\rm(A)}, the process $X:=(X_t)_{t\ge0}$ solving \eqref{stableSDE} is exponentially
	ergodic with the unique invariant probability measure $\mu$ given by \eqref{Gibbs} in the total variation norm for all $d\ge 1$ and $\alpha\in (0,2)$, and in the $L^1$-Wasserstein distance when $d\ge 1$ and $\alpha\in (1,2)$.
\end{corollary}

Let us make some comments on Assumption (A) and Theorem \ref{mainResult}, as well as the fractional drifts defined by \eqref{specialDrift} when $d>2-\alpha$ and by \eqref{e:one} when $d\le 2-\alpha$.
The most important conclusion from Theorem \ref{mainResult} is that the SDE \eqref{stableSDE} with $\alpha$-stable noise is exponentially ergodic for a large class of potentials, for which the corresponding SDE \eqref{LangevinSDE} with Brownian noise is not.

\begin{remark}
	Theorem \ref{mainResult} is concerned with
 rotationally symmetric measures $\mu$ (since $V$ is a radial function on $\R^d$).
 Condition \eqref{e:1.1} is a relatively weak condition that we need in order to prove the exponential ergodicity of the process $X$
  (indeed, it seems to be optimal as indicated by the exponential ergodicity for Ornstein–Uhlenbeck processes driven by symmetric $\alpha$-stable processes, cf. \cite{Masuda, WW}).
  It is satisfied, for example, by all potentials $V(x) = (1+|x|^2)^{\beta}$ for any $\beta>0$,
and by $V(x) = \log^{\beta}(1+|x|^2)$ for any $\beta > 1$,
  as well as by $V(x)=\beta\log(1+|x|^2)$ for any
$\beta\ge (d+\alpha)/2$.
We remark that it has been shown in \cite{RobertsTweedie1996} that for the latter two large classes of potentials,
as well as for the potentials $V(x) = (1+|x|^2)^{\beta}$ with $\beta < 1/2$,
the SDE \eqref{LangevinSDE} driven by the Brownian motion is not exponentially ergodic.  It is also easy to see that assumption (ii) for $d\le 2-\alpha$, as well as the first condition in  \eqref{e:1.2} for $d> 2-\alpha$, are
satisfied for all the potentials above.
Moreover, when $d> 2-\alpha$, we also require condition \eqref{e:logconc}, which means that the measure $\mu$ is  log-concave at infinity. The most restrictive condition is the second condition in \eqref{e:1.2}, which is essentially an assumption about sufficiently heavy tails of $\mu$ in relation to its mass in the region where $V' \le 0$, i.e., where $\mu$ is not log-concave.
In other words, if $r_0$  is not too large and if $\mu$ has heavy tails, then $\int_{r_0}^{\infty} e^{-V(r)} V'(r) r^d dr$ can be large enough so that the second condition in \eqref{e:1.2} holds. Obviously, if $\mu$ is log-concave everywhere, then the second condition in \eqref{e:1.2} is always satisfied.
\end{remark}

\begin{remark}
Let us informally discuss how the form of the fractional drifts given by \eqref{specialDrift} and  \eqref{e:one} is motivated by the requirement that the associated {\rm SDE} \eqref{stableSDE} has an invariant probability measure given by \eqref{Gibbs}. Suppose first that $d>2-\alpha$. Note that the generator of the process $X$ solving SDE \eqref{stableSDE} is
$Lf=-(-\Delta)^{\alpha/2}f+b\cdot\nabla f$.
Hence, informally, its dual operator enjoys the expression
$L^*f=-(-\Delta)^{\alpha/2}f+{\rm div}(bf)$;
see Remark \ref{e:remk--}. Roughly speaking, the density function $e^{-V(x)}$ of the
 invariant
  probability measure \eqref{Gibbs}
 is the fundamental solution to $L^*u=0$; that is,
${\rm div} (be^{-V})=-(-\Delta)^{\alpha/2}e^{-V}$. If we write $-(-\Delta)^{\alpha/2}e^{-V}= \Delta[(-\Delta)^{-(1-\alpha/2)}e^{-V}]={\rm div}\nabla [(-\Delta)^{-(1-\alpha/2)}e^{-V}]$, then a right choice for the drift can be $b(x)=e^{V(x)}\nabla(-\Delta)^{-(1-\alpha/2)}e^{-V(x)},$ which is equivalent to \eqref{specialDrift}; see the discussion in the beginning of Subsection \ref{case-1}. When $d\le 2-\alpha$, $(-\Delta)^{-(1-\alpha/2)}$ is not well defined, but we can informally write
$\nabla(-\Delta)^{-(1-\alpha/2)}=\nabla (\Delta)^{-1} [-(-\Delta)^{\alpha/2}]$ and understand $\nabla (\Delta)^{-1}$ as an integral operator. With this
in mind, we can see the intuition behind the formula for \eqref{e:one}. A fully rigorous proof that the probability measure given by \eqref{Gibbs} is invariant for \eqref{stableSDE} will be given in Proposition \ref{prop:inv}.
\end{remark}

\begin{remark}\label{remarkDiscrete} As we will see in the sequel, the drift term $b(x)$ defined by \eqref{specialDrift} when $d\ge1$ and $\alpha\in (0,1)$ or by \eqref{e:one} when $d\le 2-\alpha$, belongs to $C^1(\R^d)$; however, when $d\ge 2-\alpha$ and $\alpha\in [1,2)$, $b(x)$ defined by \eqref{specialDrift} seems to be only H\"{o}lder continuous; cf.\ Lemma \ref{L:lemma0}.
This may lead to some issues when one wants to consider discretizations of \eqref{stableSDE} in the latter case.	
When $d=1$ and $\alpha\in (1,2)$, in \cite{Simsekli2017} some numerical experiments were carried out by employing an Euler discretization of (\ref{stableSDE}) that involved approximating the drift (\ref{specialDrift}) via a series representation from \cite{Ortigueira2006}, see Section 4 and formula (7) in \cite{Simsekli2017}.
However, in order to rigorously analyse convergence of discretized \eqref{stableSDE} in this case, one cannot rely on classical results for Euler discretizations that utilize the Lipschitz property of the drift, or even results based on taming such as \cite{DareiotisKumar2016, KumarSabanis2017}, where the one-sided Lipschitz property is required. Nevertheless, there has been some recent work \cite{KS,MikuleviciusXu2018} on discretizations of L\'{e}vy-driven SDEs with bounded H\"{o}lder continuous drifts that could be applicable in our setting after an extension to the unbounded case (cf.\ Lemma \ref{L:lemma0} below for a proof of the local H\"{o}lder property of $b(x)$ given by (\ref{specialDrift})). This, however, falls beyond the scope of the present paper and will be considered in a future project.
\end{remark}

The remaining part of this paper is organised as follows.
In Section \ref{sectionProperties}, we obtain some
 explicit estimates for the fractional drift
 given by \eqref{specialDrift} when $d>2-\alpha$ and
 by \eqref{e:one} when $d\le 2-\alpha$, under Assumption (A).
 In particular, under a mild additional assumption, we get that $\langle b(x),x\rangle \asymp -\frac{e^{V(x)}}{|x|^{d+\alpha}}|x|^2$ for $|x|$ large enough. We also  claim that the fractional drift term is locally $(2-\alpha)$-H\"{o}lder continuous when
  $\alpha\in (1,2)$,
   locally $(1-\varepsilon)$-H\"{o}lder continuous for any $\varepsilon>0$ when $\alpha= 1$, and belongs to $C^1(\R^d)$ when $\alpha\in (0,1)$.
 Section \ref{sectionFurther}
 is devoted to properties of the SDE  \eqref{stableSDE} with the fractional drift terms.
We prove that the SDE  \eqref{stableSDE} with these drifts has a unique strong solution, and  show that $\mu$ given by (\ref{Gibbs}) is the unique invariant measure for
\eqref{stableSDE}.
Finally, we conclude by proving Theorem \ref{mainResult}.

\section{
Properties of the fractional drift}\label{sectionProperties}
\subsection{The case of $d> 2-\alpha$}\label{case-1}
In this subsection, we always assume that $d\ge1$ and
 $\alpha\in (0,2)$ with $d>2-\alpha$. Let $V\in C^1(\R^d)$ such that $e^{-V}|\nabla V|\in L^1(\R^d;dx)
 \cap C_b(\R^d)$. We first note that for the drift term $b(x)$ defined by \eqref{specialDrift}, it holds that
\begin{equation}\label{exa-b}b(x)=e^{V(x)}\nabla((-\Delta)^{-(1-\alpha/2)}e^{-V})(x),\end{equation} where $(-\Delta)^{-(1-\alpha/2)}$ is the Green operator corresponding to the symmetric (rotationally invariant) $(2-\alpha)$-stable process on $\R^d$, cf.\ \cite{BBKRSV, Kwa} and the references therein.
Since $d>2-\alpha$,  the symmetric $(2-\alpha)$-stable process is transient on $\R^d$, and so $(-\Delta)^{-(1-\alpha/2)}$ is well defined;
moreover,
$$(-\Delta)^{-(1-\alpha/2)} f(x)=C_{d,2-\alpha}\int_{\R^d} \frac{f(y)}{|x-y|^{d-(2-\alpha)}}\,dy,\quad f\in L^1(\R^d;dx),$$
 see \cite[Definition 2.11]{Kwa}.
Indeed, because $V\in C^1(\R^d)$ and $e^{-V}|\nabla V|\in L^1(\R^d;dx)
\cap C_b(\R^d)$, by the dominated convergence theorem, for any $x\in \R^d$,
\begin{equation}\label{ecg}
\aligned \quad& \nabla((-\Delta)^{-(1-\alpha/2)}e^{-V})(x)\\
&=C_{d,2-\alpha}\nabla\bigg[\int_{\R^d} \frac{e^{-V(y)}}{|\cdot-y|^{d-(2-\alpha)}}\,dy\bigg] (x)
=C_{d,2-\alpha}\nabla \bigg[\int_{\R^d}\frac{e^{-V(\cdot-z)}}{|z|^{d-(2-\alpha)}}\,dz\bigg](x) \\
&=- C_{d,2-\alpha}\int_{\R^d}\frac{e^{-V(x-z)}\nabla V(x-z)}{|z|^{d-(2-\alpha)}}\,dz=-C_{d,2-\alpha}\int_{\R^d}\frac{e^{-V(y)}\nabla V(y)}{|x-y|^{d-(2-\alpha)}}\,dy.
\endaligned
\end{equation}

\begin{remark}When $\alpha=2$, by \eqref{exa-b} the drift term
$b(x)$ becomes $- \nabla V(x)$. Moreover, $Z_t$ becomes $\sqrt{2}B_t$, and hence the SDE \eqref{stableSDE} is
reduced to \eqref{LangevinSDE}. \end{remark}

Recall that for any $\theta\ge0$, the H\"{o}lder-Zygmund space $\mathcal{C}_b^\theta(\R^d)$ is defined by
$$ \mathcal{C}_b^\theta(\R^d)=\left\{f\in C_b(\R^d): \|f\|_{\mathcal{C}_b^\theta(\R^d)}:=\|f\|_\infty +\sup_{x\in \R^d, h\neq 0}\frac{ \Delta_h^{[\theta]+1} f(x)}{|h|^\theta}<\infty\right\},$$ where $$\Delta_hf(x)=f(x+h)-f(x),\quad \Delta_h^jf(x)=\Delta_h(\Delta_h^{j-1}f)(x),\,\, j\ge 2.$$
Note that when $\theta\in (0,\infty)\backslash \Z_+$, $\mathcal{C}_b^\theta(\R^d)$ coincides with the classical H\"{o}lder space $C_b^\theta(\R^d)$ equipped with the norm
$$\|f\|_{C_b^\theta(\R^d)}:=\|f\|_\infty+\sum_{j=1}^{[\theta]}\sum_{\beta\in \Z_0^d \hbox{ and } |\beta|=j}\|\partial^\beta f\|_\infty+\max_{\beta\in \Z_0^d \hbox{ and }|\beta|=[\theta]}\sup_{x\neq y}\frac{|\partial^\beta f(x)-\partial^\beta f(y)|}{|x-y|^{\theta-[\theta]}},$$ where
 $\Z_+=\{1,2,\cdots\}$, $\Z_0=\Z_+\cup\{0\}$,
 $|\beta|=|\beta_1|+\cdots+|\beta_d|$ for $\beta=(\beta_1,\beta_2,\cdots,\beta_d)$; see \cite[Theorem 1 in Section 2.7.2, p.\ 201]{Tri}.
 However, when $\theta \in \Z_+$, the H\"{o}lder-Zygmund space $\mathcal{C}_b^\theta(\R^d)$ is strictly larger than $C_b^\theta(\R^d)$. In particular, when $\theta=1$, $\mathcal{C}_b^1(\R^d)$ is strictly larger than the space of bounded Lipschitz
continuous functions (see \cite[Example in Section 4.3.1, p.\ 148]{Stein}), which is, in turn, strictly larger than $C_b^1(\R^d)$. Note also that
$\mathcal{C}_b^1(\R^d)\subset C_{b}^{1-\varepsilon}(\R^d)$ for any $\varepsilon>0$.

We have the following statement.

\begin{lemma}\label{L:lemma0} Assume that
$V\in C^1(\R^d)$
 such that $e^{-V}|\nabla V|\in L^1(\R^d;dx)\cap C_b(\R^d)$. Then, the drift term $b(x)$ defined by \eqref{specialDrift}
 is locally $(2-\alpha)$-H\"{o}lder continuous when
 $\alpha\in (1,2)$, is locally $(1-\varepsilon)$-H\"{o}lder continuous for any $\varepsilon>0$ when $\alpha= 1$, and is in $C^1(\R^d)$ when $\alpha\in (0,1)$.
 \end{lemma}

\begin{proof} Suppose first that $\alpha \in (1,2)$. By $V\in C^1(\R^d)$ and $e^{-V}|\nabla V|\in L^1(\R^d;dx)
\cap C_b(\R^d)
$, it is easy to see that $b(x)$ defined by \eqref{specialDrift} is locally bounded.
 Since
 $V\in
C^1(\R^d)$, from \eqref{ecg}, to prove the desired assertion it suffices to verify that $(-\Delta)^{-(1-\alpha/2)}f\in
 \mathcal{C}_b^{2-\alpha}(\R^d)$ for all $f\in L^1(\R^d;dx)\cap B_b(\R^d)$.
Indeed, let $p(t,x,y)=p(t,x-y)$ and $(P_t)_{t\ge0}$ be the transition density function and the semigroup of the $(2-\alpha)$-symmetric stable process, respectively.
It is known that there is a constant $c_1>0$ such that
$$\|\nabla P_t f\|_\infty\le c_1t^{-1/(2-\alpha)}\|f\|_\infty,\quad t>0, f\in B_b(\R^d),$$ which is equivalent to saying that there is a constant $c_2>0$ such that for all $t>0$,
\begin{equation}\label{e:grandient}\int_{\R^d} |\nabla p(t,\cdot)(x)|\,dx\le c_2 t^{-1/(2-\alpha)};\end{equation} see \cite[Example 1.5 and Theorem 3.2]{SSW} or \cite[Lemma 4.1 and the proof of Corollary 2.5]{KS}.
Recall that, for any $f\in L^1(\R^d;dx)\cap B_b(\R^d)$,
\begin{align*}(-\Delta)^{-(1-\alpha/2)}f(x)&=C_{d,2-\alpha}\int_{\R^d}\frac{f(y)}{|x-y|^{d-(2-\alpha)}}\,dy\\
&=\int_{\R^d} f(y) \int_0^\infty p(t,x-y)\,dt\,dy=\int_0^\infty \int_{\R^d} f(y)p(t,x-y)\,dy\,dt.\end{align*}
Thus,
when $\alpha\in (1,2)$,
 for any $f\in L^1(\R^d;dx)\cap B_b(\R^d)$ and $x,h\in \R^d$,
\begin{align*}|(-\Delta)^{-(1-\alpha/2)}f (x)-&(-\Delta)^{-(1-\alpha/2)}f (x+h)|\\
&\le\|f\|_\infty\int_0^\infty\int_{\R^d} |p(t,x-y)-p(t,x+h-y)|\,dy\,dt\\
&\le \|f\|_\infty \int_0^{|h|^{2-\alpha}} \int_{\R^d} (p(t,x-y)+p(t,x+h-y))\,dy\,dt\\
&\quad + \|f\|_\infty |h|\int_{|h|^{2-\alpha}}^\infty \int_0^1\int_{\R^d} |\nabla p(t,x+\eta h-y)|\,dy\,d\eta\,dt\\
&\le 2 \|f\|_\infty|h|^{2-\alpha}+c_2\|f\|_\infty|h|\int_{|h|^{(2-\alpha)}}^\infty  t^{-1/(2-\alpha)}\,dt\\
&\le c_3\|f\|_\infty |h|^{2-\alpha},   \end{align*}
where in the last inequality we used the fact that $2-\alpha\in (0,1)$ due to $\alpha\in (1,2)$. In particular, for any $f\in L^1(\R^d;dx)\cap B_b(\R^d)$, $(-\Delta)^{-(1-\alpha/2)}f\in
\mathcal{C}^{2-\alpha}_b(\R^d)=
C^{2-\alpha}_b(\R^d)$.

Next, we consider the case of $\alpha\in(0,1]$.
According to \eqref{e:grandient} and \cite[Lemma 4.1(3)]{KS} as well as the iterating procedure, there is a constant $c_4>0$ such that for all $t>0$,
$$\int_{\R^d}
|\nabla^2 p(t,\cdot)(x)|\,dx\le c_4 t^{-2/(2-\alpha)}. $$
Then, for any $f\in L^1(\R^d;dx)\cap B_b(\R^d)$ and $x,h\in \R^d$,
\begin{align*}&|\Delta^2_h(-\Delta)^{-(1-\alpha/2)}f(x)|\\
&=|(-\Delta)^{-(1-\alpha/2)}f (x+2h)-2(-\Delta)^{-(1-\alpha/2)}f (x+h)+(-\Delta)^{-(1-\alpha/2)}f (x)|\\
&\le\|f\|_\infty\int_0^\infty\int_{\R^d} |p(t,x+2h-y)-2p(t,x+h-y)+p(t,x-y)|\,dy\,dt\\
&\le \|f\|_\infty \int_0^{|h|^{2-\alpha}} \int_{\R^d} (p(t,x+2h-y)+2p(t,x+h-y) +p(t,x-y))\,dy\,dt\\
&\quad + c_5\|f\|_\infty |h|^2\int_{|h|^{2-\alpha}}^\infty \int_0^1(1-\eta) \int_{\R^d} |\nabla^2 p(t,x+\eta h-y)|\,dy\,d\eta\,dt\\
&\le 4 \|f\|_\infty|h|^{2-\alpha}+  c_6\|f\|_\infty|h|^2\int_{|h|^{2-\alpha}}^\infty  t^{-2/(2-\alpha)}\,dt\\
&\le c_7\|f\|_\infty |h|^{2-\alpha},   \end{align*} where in the
second inequality we used the Taylor formula.
Hence,
$(-\Delta)^{-(1-\alpha/2)}f\in
 \mathcal{C}_b^{2-\alpha}(\R^d)$, thanks to the fact that $(-\Delta)^{-(1-\alpha/2)}f$ is bounded for any $f\in L^1(\R^d;dx)\cap B_b(\R^d)$.
The proof is completed. \end{proof}

\begin{remark}\label{remark:notC1}
From expression \eqref{specialDrift}, one may expect  that the drift term $b(x)$ does not belong to $C^1(\R^d)$
when $\alpha\in (1,2)$. Informally,
since the integral
$$\int_{\R^d}\frac{|f(y)|}{|x-y|^{d-(2-\alpha)+1}}\,dy$$
 may diverge for $f\in L^1(\R^d;dx)\cap B_b(\R^d)$ with $\alpha\in (1,2)$, we cannot take the derivative inside the integral in \eqref{specialDrift}.
 \end{remark}

In the rest of this part, we will further assume that $V$ is a radial function. We will present some explicit estimates for the drift term $b(x)$ defined by  \eqref{specialDrift}, i.e.,
$$b(x) = -C_{d,2-\alpha} e^{V(x)} \int_{\R^d} \frac{e^{-V(y)} \nabla V(y)}{|x-y|^{d-(2-\alpha)}}\, dy
= -C_{d,2-\alpha} e^{V(|x|)} \int_{\R^d} \frac{e^{-V(|y|)}  V'(|y|)y}{|y||x-y|^{d-(2-\alpha)}}\, dy \,.$$ In particular, it holds that
$b(x)=-b(-x)$ and $b(0)=0$, i.e., $b(x)$ is an anti-symmetric function on $\R^d$.

With a slight abuse of notation, in the following we write $V(x)=V(|x|)$ for all $x\in \R^d$.

\begin{lemma}\label{Lemma3.1} Let $V(x)=V(|x|)$ for all $x\in \R^d$ such that
$V\in
C^1(\R^d)$ and $e^{-V}|\nabla V|\in L^1(\R^d;dx)\cap C_b(\R^d)$.
  Suppose that $r_0:=\sup\{r>0: V'(r)\le 0\}<\infty,$
 \begin{equation}\label{e:ii-}\int_{0}^\infty e^{-V(r)} |V'(r)| r^d\,dr<\infty\end{equation} and
\begin{equation}\label{e:cond-33}\int_0^\infty e^{-V(r)} V'(r) r^d\,dr >0.\end{equation}
Then, there exist constants $c_1, c_2>0$
and $r_1\ge1$
 such that for all $x\in \R^d$,
\begin{equation}\label{e:assertion3-1}
\langle x, b(x)\rangle\le c_1
\I_{\{|x|\le r_1\}}
- \frac{c_2 e^{V(|x|)}}{(1+|x|)^{d+\alpha}}|x|^2
\I_{\{|x|>r_1\}}.
\end{equation}
 \end{lemma}

\begin{proof} For any $x\in \R^d$, by changing the variables, we find that
\begin{align*}C_{d,2-\alpha}^{-1}\langle x, b(x)\rangle=& -e^{V(|x|)} \int_{\R^d}\frac{e^{-V(|y|)} V'(|y|)\langle y,x\rangle}{|y||x-y|^{d-(2-\alpha)}}\,dy\\
=&-e^{V(|x|)} \int_{\{\langle x,y\rangle\ge0\}}\frac{e^{-V(|y|)} V'(|y|)\langle y,x\rangle}{|y|}\left(\frac{1}{|x-y|^{d-(2-\alpha)}}-\frac{1}{|x+y|^{d-(2-\alpha)}}\right)\,dy\\
=&-e^{V(|x|)} \int_{\{ V'(|y|)\le 0, \langle x,y\rangle\ge0\}}\frac{e^{-V(|y|)} V'(|y|)\langle y,x\rangle}{|y|}\left(\frac{1}{|x-y|^{d-(2-\alpha)}}-\frac{1}{|x+y|^{d-(2-\alpha)}}\right)\,dy\\
&-e^{V(|x|)} \int_{\{ V'(|y|)\ge 0, \langle x,y\rangle\ge0\}}\frac{e^{-V(|y|)} V'(|y|)\langle y,x\rangle}{|y|}\left(\frac{1}{|x-y|^{d-(2-\alpha)}}-\frac{1}{|x+y|^{d-(2-\alpha)}}\right)\,dy\\
=&:J_1+J_2.\end{align*}

Note that, for any $x,y\in \R^d$, we have
\begin{align*}&\frac{1}{|x-y|^{d-(2-\alpha)}}- \frac{1}{|x+y|^{d-(2-\alpha)}}
 =\left(|x|^2+|y|^2-2\langle x,y\rangle\right)^{-(d+\alpha-2)/2}- \left(|x|^2+|y|^2+2\langle x,y\rangle\right)^{-(d+\alpha-2)/2},
\end{align*}
and that for the function $\psi(r) := r^{-{(d+\alpha - 2)}/{2}}$, we have $$\psi(r-\delta) - \psi(r+\delta) \le -2\delta \psi'(r-\delta),\quad 0\le \delta \le r,$$ thanks to
$\psi'' \ge 0$ and the mean value theorem. Hence, taking $r = |x|^2 + |y|^2$ and $\delta = 2 \langle x , y \rangle \ge 0$, we get
\begin{align*}J_1 \le &- 2(d+\alpha-2) e^{V(|x|)} \int_{\{ V'(|y|) \le 0 , \langle x , y \rangle \ge 0 \}} \frac{e^{-V(|y|)} V'(|y|) \langle y , x \rangle^2 }{|y| } \frac{1}{|x-y|^{d+\alpha}} \, dy\\
\le&- 2(d+\alpha-2) e^{V(|x|)} \int_{\{ V'(|y|) \le 0 , \langle x , y \rangle \ge 0 \}} \frac{e^{-V(|y|)} V'(|y|) \langle y , x \rangle^2 }{|y| } \frac{1}{||x|-|y||^{d+\alpha}} \, dy\\
= &- (d+\alpha-2) e^{V(|x|)}   \int_{\{ V'(|y|) \le 0  \}} \frac{e^{-V(|y|)} V'(|y|) \langle y , x \rangle^2 }{|y| } \frac{1}{||x|-|y||^{d+\alpha}} \, dy\\
=& - (d+\alpha-2) e^{V(|x|)}   |x|^2\int_{\{ V'(|y|) \le 0  \}} \frac{e^{-V(|y|)} V'(|y|) y_1^2 }{|y| } \frac{1}{||x|-|y||^{d+\alpha}} \, dy\\
=&-\frac{d+\alpha-2}{d} e^{V(|x|)}|x|^2 \int_{\{V'(|y|)\le0\}}\frac{e^{-V(|y|)} V'(|y|)|y|}{||x|-|y||^{d+\alpha}} \, dy.
\end{align*}
Since $r_0:=\sup\{r>0: V'(r)\le 0\}<\infty$, for any $C>1$ and any $x\in \R^d$ with $|x|\ge Cr_0$,
we have
$$J_1\le -\frac{d+\alpha-2}{d}(1-C^{-1})^{-d-\alpha} \frac{e^{V(|x|)} |x|^2}{|x|^{d+\alpha}}\int_{\{ V'(|y|) \le 0  \}}  {e^{-V(|y|)} V'(|y|) |y|}\,dy.$$

On the other hand, for any $x,y\in \R^d$ with $\langle x,y\rangle\ge0$,
\begin{equation}\label{e:jj}
\frac{1}{|x-y|^{d-(2-\alpha)}}- \frac{1}{|x+y|^{d-(2-\alpha)}} \ge 2(d+\alpha-2)(|x|^2+|y|^2)^{-(d+\alpha)/2} \langle x,y\rangle.
\end{equation}
Here we used the fact that for the function $\psi(r)=r^{-(d+\alpha-2)/2}$, it holds that
$$\psi(r-\delta)-\psi(r+\delta)\ge -2 \psi'(r)\delta,\quad 0\le \delta\le r,$$ thanks to the mean value theorem again and the fact that $\psi'''\le 0$.
Combining \eqref{e:jj} with the fact
 $V'(r)\ge0$ for all $r\ge r_0$, we get that for any  $a>0$  and any $x\in \R^d$,
\begin{align*}J_2&\le -2(d+\alpha-2) e^{V(|x|)}\int_{\{V'(|y|) \ge0, \langle y,x\rangle\ge0\}}\frac{e^{-V(|y|)}V'(|y|)\langle x,y\rangle^2}{|y|(|x|^2+|y|^2)^{(d+\alpha)/2}}\,dy\\
&= -(d+\alpha-2) e^{V(|x|)}\int_{\{V'(|y|) \ge0\}}\frac{e^{-V(|y|)}V'(|y|)\langle x,y\rangle^2}{|y|(|x|^2+|y|^2)^{(d+\alpha)/2}}\,dy\\
&= -(d+\alpha-2) e^{V(|x|)}\int_{\{V'(|y|) \ge0\}}\frac{e^{-V(|y|)}V'(|y|)|x|^2 y_1^2}{|y|(|x|^2+|y|^2)^{(d+\alpha)/2}}\,dy\\
&=-\frac{d+\alpha-2}{d} e^{V(|x|)}|x|^2\int_{\{V'(|y|) \ge0\}}\frac{e^{-V(|y|)}V'(|y|)|y|}{(|x|^2+|y|^2)^{(d+\alpha)/2}}\,dy\\
&\le -\frac{d+\alpha-2}{d} e^{V(|x|)}|x|^2\int_{\{V'(|y|) \ge0, |y|\le a|x|\}}\frac{e^{-V(|y|)}V'(|y|)|y|}{(|x|^2+|y|^2)^{(d+\alpha)/2}}\,dy\\
&\le -\frac{d+\alpha-2}{d}(1+a^2)^{-(d+\alpha)/2} \frac{e^{V(|x|)}|x|^2}{|x|^{d+\alpha}}\int_{\{V'(|y|) \ge0, |y|\le a|x|\}}e^{-V(|y|)}V'(|y|)|y|\,dy. \end{align*}

According to both estimates above for $J_1$ and $J_2$, we find that for any $x\in \R^d$ with $|x|\ge Cr_0$,
\begin{align*} C_{d,2-\alpha}^{-1}\langle x, b(x)\rangle
&\le -(1-C^{-1})^{-d-\alpha} \Big(\frac{d+\alpha-2}{d}\Big) \frac{e^{V(|x|)}|x|^2}{|x|^{d+\alpha}} \\
&\quad\times \bigg[\frac{(1+a^2)^{-(d+\alpha)/2}}{(1-C^{-1})^{-d-\alpha}}\int_{\{V'(|y|) \ge0, |y|\le a|x|\}}e^{-V(|y|)}V'(|y|)|y|\,dy\\
&\quad\qquad+\int_{\{ V'(|y|) \le 0  \}}  {e^{-V(|y|)} V'(|y|) |y|}\,dy\bigg].\end{align*}

Note that, under \eqref{e:cond-33}, $$\int_{\R^d}e^{-V(|y|)}V'(|y|)|y|\,dy>0.$$ Then,
by \eqref{e:ii-},
there is a constant $R_0>r_0$ such that
$$\int_{\{|y|\le R_0\}}e^{-V(|y|)}V'(|y|)|y|\,dy>0.$$ This implies that
\begin{equation}\label{e:key1}\int_{\{V'(|y|) \ge0, |y|\le R_0\}}e^{-V(|y|)}V'(|y|)|y|\,dy+ \int_{\{ V'(|y|) \le 0  \}}  {e^{-V(|y|)} V'(|y|) |y|}\,dy>0,\end{equation} where we used the facts that $r_0:=\sup\{r>0: V'(r)\le 0\}<\infty$
and $r_0 < R_0$.
Furthermore, by  \eqref{e:key1}, we can choose $\varepsilon\in (0,1)$ small enough so that
$$M:=(1-\varepsilon)\int_{\{V'(|y|) \ge0, |y|\le R_0\}}e^{-V(|y|)}V'(|y|)|y|\,dy + \int_{\{ V'(|y|) \le 0  \}}  {e^{-V(|y|)} V'(|y|) |y|}\,dy>0.$$
Now for these fixed $R_0$ and $\varepsilon$, we find $C>1$ large enough and $a>0$ small enough such that
$$\frac{(1+a^2)^{-(d+\alpha)/2}}{(1-C^{-1})^{-d-\alpha}}\ge 1-\varepsilon,\quad aCr_0\ge R_0.$$ Then, for any $x\in \R^d$ with $|x|\ge Cr_0$,
\begin{align*}&\frac{(1+a^2)^{-(d+\alpha)/2}}{(1-C^{-1})^{-d-\alpha}}\int_{\{V'(|y|) \ge0, |y|\le a|x|\}}e^{-V(|y|)}V'(|y|)|y|\,dy +\int_{\{ V'(|y|) \le 0  \}}  {e^{-V(|y|)} V'(|y|) |y|}\,dy\\
&\ge (1-\varepsilon)\int_{\{V'(|y|) \ge0, |y|\le aCr_0\}}e^{-V(|y|)}V'(|y|)|y|\,dy +\int_{\{ V'(|y|) \le 0  \}}  {e^{-V(|y|)} V'(|y|) |y|}\,dy\\
&\ge (1-\varepsilon)\int_{\{V'(|y|) \ge0, |y|\le R_0\}}e^{-V(|y|)}V'(|y|)|y|\,dy +\int_{\{ V'(|y|) \le 0  \}}  {e^{-V(|y|)} V'(|y|) |y|}\,dy=M>0 \end{align*}
and so \begin{equation}\label{e:conlu-1}\langle x, b(x)\rangle\le -\frac{C_{d,2-\alpha}M(1-C^{-1})^{-d-\alpha}({d+\alpha-2})}{d}\frac{e^{V(|x|)}|x|^2}{|x|^{d+\alpha}}.\end{equation}

Furthermore, by $V\in C^1(\R^d)$ and $e^{-V}|\nabla V|\in L^1(\R^d;dx)\cap C_b(\R^d)$, $b(x)$ is locally bounded; see Lemma \ref{L:lemma0}. Then, for any $x\in \R^d$ with $|x|\le l$, one can find a constant $C(l)>0$ such that $|b(x)|\le C(l)$, and so
\begin{equation}\label{e:conlu-2}\langle x,b(x)\rangle\le |x| |b(x)|\le l C(l).\end{equation}

Therefore, by \eqref{e:conlu-1} and \eqref{e:conlu-2},
we can choose the constants $c_1, c_2>0$ and $r_1>1$ so that \eqref{e:assertion3-1} holds.
\end{proof}

The following statement indicates that the estimate \eqref{e:assertion3-1} for $|x|$ large enough is indeed optimal, under a mild additional assumption.

\begin{lemma}\label{Lemma3.2}
Let $V(x)=V(|x|)$ for all $x\in \R^d$ such that $V\in
C^1(\R^d)$, $e^{-V}|\nabla V|\in L^1(\R^d;dx)\cap C_b(\R^d)$, and \eqref{e:ii-} is satisfied. If \begin{equation}\label{e:i-}\limsup_{r\to\infty} [e^{-V(r)}V'(r)r^{d+1}]<\infty,\end{equation}
  then there exists a constant $c>0$ such that for all $x\in\R^d$,
$$
|b(x)|\le \frac{ce^{V(|x|)}}{(1+|x|)^{d+\alpha-1}}.
$$
\end{lemma}

\begin{proof}
	For convenience, we set $\tilde{b}(x)=C_{d,2-\alpha}^{-1}e^{-V(|x|)}b(x)$. Then, for any $x\in \R^d$,
	$$
		|\tilde{b}(x)|^2=\sum_{i=1}^d\Big(\int_{\R^d}\frac{e^{-V(|y|)}V'(|y|)y_i}{|y||x-y|^{d-(2-\alpha)}}\,dy\Big)^2=:\sum_{i=1}^d \text{I}_i.
	$$
	For fixed $i$, assume that $x_i\ge0$. Then,
	\begin{align*}
	\text{I}_i&=\Big(\int_{\{y_i>0\}}\frac{e^{-V(|y|)}V'(|y|)y_i}{|y|}\bigg[\frac{1}{|x-y|^{d+\alpha-2}}-\frac{1}{(|x-y|^2+4x_iy_i)^{(d+\alpha-2)/2}}\bigg]\,dy\Big)^2\\
&\le\Big(\int_{\{y_i>0\}}\frac{e^{-V(|y|)}|V'(|y|)|y_i}{|y|}\bigg[\frac{1}{|x-y|^{d+\alpha-2}}-\frac{1}{(|x-y|^2+4x_iy_i)^{(d+\alpha-2)/2}}\bigg]\,dy\Big)^2\\
	&\le 3 \Big(\int_{\{y_i>0,|x-y|\le|x|/2\}}\frac{e^{-V(|y|)}|V'(|y|)|y_i}{|y|}\bigg[\frac{1}{|x-y|^{d+\alpha-2}}+\frac{1}{(|x-y|^2+4x_iy_i)^{(d+\alpha-2)/2}}\bigg]\,dy\Big)^2\\
	&\quad+3\Big(\int_{\{y_i>0,|x-y|\ge2|x|\}}\frac{e^{-V(|y|)}|V'(|y|)|y_i}{|y|}\bigg[\frac{1}{|x-y|^{d+\alpha-2}}+\frac{1}{(|x-y|^2+4x_iy_i)^{(d+\alpha-2)/2}}\bigg]\,dy\Big)^2\\
	&\quad+3\Big(\int_{\{y_i>0,|x|/2\le|x-y|\le2|x|\}}\frac{e^{-V(|y|)}|V'(|y|)|y_i}{|y|}\bigg[\frac{1}{|x-y|^{d+\alpha-2}}-
\frac{1}{(|x-y|^2+4x_iy_i)^{(d+\alpha-2)/2}}\bigg]dy\Big)^2 \\
	&\le 6\Big(\int_{\{|x-y|\le|x|/2\}}\frac{e^{-V(|y|)}|V'(|y|)||y_i|}{|y|}\frac{1}{|x-y|^{d+\alpha-2}}\,dy\Big)^2\\
	&\quad+6\Big(\int_{\{|x-y|\ge2|x|\}}\frac{e^{-V(|y|)}|V'(|y|)||y_i|}{|y|}\frac{1}{|x-y|^{d+\alpha-2}}\,dy\Big)^2\\
	&\quad+ 3 \Big(\int_{\{y_i>0,|x|/2\le|x-y|\le2|x|\}}\frac{e^{-V(|y|)}|V'(|y|)|y_i}{|y|}\bigg[\frac{1}{|x-y|^{d+\alpha-2}}-\frac{1}{(|x-y|^2+4x_iy_i)^{(d+\alpha-2)/2}}\bigg]\,dy\Big)^2 \\
		&=:6\text{I}_{i1}+6\text{I}_{i2}+3\text{I}_{i3}.
\end{align*}
When $x_i<0$, similarly we have
\begin{align*}
	\text{I}_i&=\Big(\int_{\{y_i<0\}}\frac{e^{-V(|y|)}V'(|y|)y_i}{|y|}\bigg[\frac{1}{|x-y|^{d+\alpha-2}}-\frac{1}{(|x-y|^2+4x_iy_i)^{(d+\alpha-2)/2}}\bigg]\,dy\Big)^2\\
&\le\Big(\int_{\{y_i<0\}}\frac{e^{-V(|y|)}|V'(|y|)||y_i|}{|y|}\bigg[\frac{1}{|x-y|^{d+\alpha-2}}-\frac{1}{(|x-y|^2+4x_iy_i)^{(d+\alpha-2)/2}}\bigg]\,dy\Big)^2\\
	&\le6{\rm I}_{i1}+6{\rm I}_{i2}\\
&\quad+3\Big(\int_{\{y_i<0,|x|/2\le|x-y|\le2|x|\}}
\frac{e^{-V(|y|)}|V'(|y|)||y_i|}{|y|}\bigg[\frac{1}{|x-y|^{d+\alpha-2}}-\frac{1}{(|x-y|^2+4x_iy_i)^{(d+\alpha-2)/2}}\bigg]\,dy\Big)^2\\
	&=:6\text{I}_{i1}+6\text{I}_{i2}+3 \tilde{\text{I}} _{i3}.
	\end{align*}

Next, we estimate the above terms respectively. For $\text{I}_{i1}$, we have that for $x\in \R^d$ with $|x|$ large enough,
\begin{align*}
\Big(\sum_{i=1}^d \text{I}_{i1}\Big)^{1/2}&\le\sqrt{d}\int_{\{|x-y|\le |x|/2\}}e^{-V(|y|)}|V'(|y|)|\frac{1}{|x-y|^{d-(2-\alpha)}}dy\\
&\le \sqrt{d}\sup_{\{|x-y|\le|x|/2\}}\{e^{-V(|y|)}|V'(|y|)|\}\int_{|x-y|\le |x|/2}\frac{1}{|x-y|^{d-(2-\alpha)}}dy\\
&\le \frac{c_1}{|x|^{\alpha-2}}\sup_{|y|\ge|x|/2}\{e^{-V(|y|)}|V'(|y|)|\} \le\frac{c_12^{d+1}}{|x|^{d+\alpha-1}}\sup_{|y|\ge|x|/2}\{e^{-V(|y|)}|V'(|y|)||y|^{d+1}\} \\
&\le  \frac{c_2}{|x|^{d+\alpha-1}},
	\end{align*}
where the last inequality follows from \eqref{e:i-}.

For $\text{I}_{i2}$, we have that for $x\in \R^d$ with $|x|$ large enough,
\begin{align*}
\Big(\sum_{i=1}^d \text{I}_{i2}\Big)^{1/2}&\le \sqrt{d}\int_{\{|x-y|\ge 2|x|\}}e^{-V(|y|)}|V'(|y|)|\frac{1}{|x-y|^{d-(2-\alpha)}}\,dy\\
&\le  \frac{\sqrt{d}2^{2-d-\alpha}}{|x|^{d-(2-\alpha)}}\int_{\{|y|\ge |x|\}}e^{-V(|y|)}|V'(|y|)|\,dy \le  \frac{c_3}{|x|^{d-(2-\alpha)}}\int_{\{|y|\ge |x|\}}|y|^{-d-1}\,dy\\
&\le  \frac{c_4}{|x|^{d+\alpha-1}},
	\end{align*} where in the third inequality we used \eqref{e:i-} again.

To estimate $\text{I}_{i3}$, define
$$
f(r)=\frac{1}{(|x-y|^2+r)^{(d+\alpha-2)/2}},\quad r\ge 0.
$$
By the Lagrange mean value theorem, for any $y\in\R^d$ with $|x|/2\le |x-y|\le 2|x|$ and $y_i>0$, and any $x_i\ge0$, there exists  $\theta_i\in [0,4x_iy_i]$ such that
\begin{align*}
f(0)-f(4x_iy_i)&=-4x_iy_if'(\theta_i) =\frac{d+\alpha-2}{2}\frac{4x_iy_i}{(|x-y|^2+\theta_i)^{(d+\alpha)/2}}\\
&\le 2(d+\alpha-2)\frac{|x|y_i}{|x-y|^{d+\alpha}} \le \frac{c_5y_i}{|x|^{d+\alpha-1}}.
	\end{align*} Note that it always holds that $f(0)-f(4x_iy_i)>0$.
Therefore, for all $x\in \R^d$, according to \eqref{e:ii-},
\begin{align*}
{\text{I}}_{i3}\le & \frac{c_6}{|x|^{2(d+\alpha-1)}}\Big(\int_{\{y_i>0,
	|x|/2\le |x-y|\le2|x|\}}e^{-V(|y|)}|V'(|y|)||y|\,dy\Big)^2\\
\le &\frac{c_6}{|x|^{2(d+\alpha-1)}}\Big(\int_{\{
	|x|/2\le |x-y|\le2|x|\}}e^{-V(|y|)}|V'(|y|)||y|\,dy\Big)^2\\
\le &\frac{c_6}{|x|^{2(d+\alpha-1)}}\Big(\int_{\{
 |y|\le3|x|\}}e^{-V(|y|)}|V'(|y|)||y|\,dy\Big)^2\le  \frac{c_7}{|x|^{2(d+\alpha-1)}}\end{align*} and so
$$ \Big(\sum_{i=1}^d \text{I}_{i3}\Big)^{1/2}\le \frac{c_8}{|x|^{d+\alpha-1}}.$$

Similarly, we also can prove that for all $x\in \R^d$,
$$
 \Big(\sum_{i=1}^d \tilde{\text{I}}_{i3}\Big)^{1/2}\le\frac{c_9}{|x|^{d+\alpha-1}}.
$$

Combining all the estimates above, we can obtain that there exists a constant $c>0$ such that for all $x\in \R^d$ with $|x|\ge 1$ large enough,
$$
|\tilde{b}(x)|\le  \frac{c}{|x|^{d+\alpha-1}};
$$
that is, $$|b(x)|\le \frac{ce^{V(|x|)}}{|x|^{d+\alpha-1}}.$$ The proof is completed, since $b(x)$ is locally bounded.
\end{proof}

\begin{remark} (1) If condition \eqref{e:i-} is strengthened into $\limsup\limits_{r\to\infty} [e^{-V(r)}V'(r)r^{d+1}]=0,$ then both terms $\Big(\sum_{i=1}^d \text{I}_{i1}\Big)^{1/2}$ and $\Big(\sum_{i=1}^d \text{I}_{i2}\Big)^{1/2}$ are $o\left(\frac{1}{|x|^{d+\alpha-1}}\right)$ for all $x\in \R^d$ with $|x|$ large enough. Hence, the remaining term  $\Big(\sum_{i=1}^d {\text{I}}_{i3}\Big)^{1/2}$  or $\Big(\sum_{i=1}^d \tilde{\text{I}}_{i3}\Big)^{1/2}$ plays the lead role in the estimates above.

(2) Under the assumptions of Lemmas \ref{Lemma3.1} and \ref{Lemma3.2}, it holds that, for $|x|$ large enough,
$$\langle x, b(x)\rangle \asymp -\frac{e^{V(|x|)}}{(1+|x|)^{d+\alpha}}|x|^2.$$    \end{remark}

\subsection{The case of $d\le 2-\alpha$}\label{sect22}
In this part, we will consider the case of $d\le 2-\alpha$, i.e., $d=1$ and $0<\alpha\le1$.
Let $V \in C^2(\R)$ be such that $e^{-V}\in L^1(\R;dx)\cap C_b^2(\R)$, and let $b(x)$ be defined by \eqref{e:one}.
We first show that
\begin{lemma}\label{Lemma:one1} Let $V \in C^2(\R)$ be such that $e^{-V}\in L^1(\R;dx)\cap C_b^2(\R)$. If \begin{equation}\label{cod-111}
  \limsup_{|x|\rightarrow \infty}[|x|^3e^{-V(x)}|V'(x)^2-V''(x)|]<\infty,
\end{equation}
 then $b(x)$ given by \eqref{e:one} is well defined. \end{lemma}

\begin{proof} Since $e^{-V}\in C_b^2(\R)$, we know that $-(-\Delta)^{\alpha/2}e^{-V(x)}\in C_b(\R)$, and so $-(-\Delta)^{\alpha/2}e^{-V(x)}$ is locally integrable on $\R$. Next, we will estimate $(-\Delta)^{\alpha/2}e^{-V(x)}$ for $x<-1$ small enough.
For $x<-1$,
\begin{align*}
-(-\Delta)^{\alpha/2}e^{-V(x)}&=\int_{\R}\big(e^{-V(x+z)}-e^{-V(x)}+e^{-V(x)}V'(x)z\I_{\{|z|\le 1\}}\big)\frac{c_{1,\alpha}}{|z|^{1+\alpha}}\,dz\\
& =\int_{\{|z|<-x/2\}}\big(e^{-V(x+z)}-e^{-V(x)}+e^{-V(x)}V'(x)z\big)\frac{c_{1,\alpha}}{|z|^{1+\alpha}}\,dz\\
&\quad +\int_{\{|z|\ge -x/2\}}(e^{-V(x+z)}-e^{-V(x)})\frac{c_{1,\alpha}}{|z|^{1+\alpha}}\,dz\\
&=:I_1(x)+I_2(x).
\end{align*}

Since
\begin{align*}|I_2(x)|\le & \int_{\{|z|\ge -x/2\}} e^{-V(x+z)} \frac{c_{1,\alpha}}{|z|^{1+\alpha}}\,dz+ \int_{\{|z|\ge -x/2\}} e^{-V(x)} \frac{c_{1,\alpha}}{|z|^{1+\alpha}}\,dz\\
\le&c_1\left( |x|^{-1-\alpha}+ e^{-V(x)}\right),\end{align*}
 by
$e^{-V}\in L^1(\R;dx)$ we know that $\int_{\R}|I_2(x)|\,dx<\infty$.
 On the other hand, by the mean value theorem,
 \begin{equation}\label{e:refer}\begin{split}
|I_1(x)|&\le \int_{\{|z|< -x/2\}}|e^{-V(x+z)}-e^{-V(x)}+e^{-V(x)}V'(x)z|\frac{c_{1,\alpha}}{|z|^{1+\alpha}}\,dz\\
&\le
c_{1,\alpha}\Big[\sup_{ 3x/2\le u\le
 x/2}|e^{-V(u)}(V'(u)^2-V''(u))|\Big]\int_0^{-x/2}z^{1-\alpha}\,dz\\
&\le c_2(- x)^{2-\alpha}
\sup_{ 3x/2\le u\le
 x/2} \big[ e^{-V(u)}|V'(u)^2-V''(u)|\big] \\
&\le 8 c_2(- x)^{-1-\alpha} \sup_{ 3x/2\le u\le
	x/2} \big[|u|^3 e^{-V(u)}|V'(u)^2-V''(u)|\big] \\
& \le c_3 (-x)^{-1-\alpha},
\end{split}\end{equation} where in the last inequality we used \eqref{cod-111}.
Note that analogous estimates hold also for $x > 1$ large enough, and hence we arrive at the desired assertion.
 \end{proof}
\begin{remark}\label{Rem-00}
From the proof above, we can see that under the assumptions of Lemma \ref{Lemma:one1},
\begin{equation}\label{e:integrability}
\int_{\R}| (-\Delta)^{\alpha/2} e^{-V(u)}|\,du<\infty
\end{equation}
and hence
$$\int_{x}^\infty (-\Delta)^{\alpha/2} e^{-V(u)}\,du$$ is also well defined for any $x \in \R$.
\end{remark}

In the following, we always assume that \eqref{cod-111} holds. We further suppose that $V(x)=V(-x)$ for all $x\in \R$. Then, we claim that
\begin{lemma}\label{Lemma:second2}  Let $V \in C^2(\R)$ be such that $e^{-V}\in L^1(\R;dx)\cap C_b^2(\R)$. Suppose that \eqref{cod-111} holds and that $V(x)=V(-x)$ for all $x\in \R$. Then, $b(x)$ given by \eqref{e:one} is an anti-symmetric function on $\R$ $($i.e., $b(x)=-b(-x)$ for all $x\in \R$$)$ such that
\begin{equation}\label{e:driftd1}
b(x)=\begin{cases} e^{V(x)}\displaystyle\int_x^\infty (-\Delta)^{\alpha/2}e^{-V(z)}\,dz,&\quad x\ge0,\\
-e^{V( x)}\displaystyle\int_{-x}^\infty (-\Delta)^{\alpha/2}e^{-V(z)}\,dz,&\quad x<0.\end{cases}
\end{equation}
In particular,
 $b(0)=0$.
 Moreover, $b(x) \in C^1(\R)$ and is locally bounded. \end{lemma}

\begin{proof}
As mentioned in Remark \ref{Rem-00}, under the assumptions of this lemma,
we have \eqref{e:integrability}.
We will show that this yields
 \begin{equation}\label{e:22aux1}
 -\int_\R (-\Delta)^{\alpha/2}e^{-V(u)}\,du=0,
 \end{equation}
  and hence
 $$b(x)=-e^{V(x)}\int_{-\infty}^x (-\Delta)^{\alpha/2} e^{-V(u)}\,du =e^{V(x)}\int_{x}^\infty (-\Delta)^{\alpha/2} e^{-V(u)}\,du,\quad x\ge 0.$$
 Indeed,
for any $\varepsilon\in (0,1]$ and any $x\in \R$,
 \begin{align*}&\left|\int_{\{|y-x|\ge \varepsilon\}} \frac{(e^{-V(y)}-e^{-V(x)})}{|y-x|^{1+\alpha}}\,dy\right|\\
 &= \left|\int_{\{|z|\ge \varepsilon\}} \big(e^{-V(x+z)}-e^{-V(x)}+e^{-V(x)}V'(x)z\I_{\{|z|\le 1\}}\big)\frac{dz}{|z|^{1+\alpha}}\right|\\
 &\le  \int_{\{ |z|\le 1\}} \big|e^{-V(x+z)}-e^{-V(x)}+e^{-V(x)}V'(x)z\big|\frac{dz}{|z|^{1+\alpha}}  +\left|\int_{\{ |z|>1\}} \big(e^{-V(x+z)}-e^{-V(x)}\big)\frac{dz}{|z|^{1+\alpha}}\right|\\
 &\le\frac{1}{2} \|[e^{-{V}}]''\|_\infty\int_{\{|z|\le 1\}} \frac{|z|^2}{|z|^{1+\alpha}}\,dz + 2\|e^{-V}\|_\infty \int_{\{|z|>1\}} \frac{1}{|z|^{1+\alpha}}\,dz\\
 &\le c_1<\infty. \end{align*} On the other hand, for any $\varepsilon\in (0,1]$ and any $x\in \R$ with $|x|>2$ large enough,
  \begin{align*}&\left|\int_{\{|y-x|\ge \varepsilon\}} \frac{(e^{-V(y)}-e^{-V(x)})}{|y-x|^{1+\alpha}}\,dy\right|\\
  &=\left|\int_{\{|z|\ge \varepsilon\}}\big(e^{-V(x+z)}-e^{-V(x)}+e^{-V(x)}V'(x)z\I_{\{|z|\le |x|/2\}}\big)\frac{1}{|z|^{1+\alpha}}\,dz\right|\\
   &\le \int_{\{|z|\le |x|/2\}} \big|e^{-V(x+z)}-e^{-V(x)}+e^{-V(x)}V'(x)z\big| \frac{1}{|z|^{1+\alpha}}\,dz\\
  &\quad+\int_{\{|z|>|x|/2\}} e^{-V(x+z)}\frac{1}{|z|^{1+\alpha}}\,dz+e^{-V(x)}\int_{\{|z|>|x|/2\}}\frac{1}{|z|^{1+\alpha}}\,dz\\
  &\le c_2|x|^{-(1+\alpha)}+\frac{2^{1+\alpha}}{|x|^{1+\alpha}}\int_{\R}e^{-V(z)}\,dz+c_3 e^{-V(x)}, \end{align*}
where the first term in the last inequality follows from \eqref{cod-111} and the argument for \eqref{e:refer}.
Hence, there is a constant $c_4>0$
 such that for all $x\in \R$,  $$\sup_{\varepsilon\in (0,1]}\left|\int_{\{|y-x|\ge \varepsilon\}} \frac{(e^{-V(y)}-e^{-V(x)})}{|y-x|^{1+\alpha}}\,dy\right|\le c_4\left((1+|x|)^{-1-\alpha}+e^{-V(x)}\right).$$
Therefore,
 by
 using the dominated convergence theorem and
 changing
 the order of integration,
 we find that
 \begin{align*} -\int_\R (-\Delta)^{\alpha/2}e^{-V(x)}\,dx=&  c_{1,\alpha} \int_{\R}\lim_{\varepsilon \to0}\int_{\{|y-x|\ge \varepsilon\}} \frac{(e^{-V(y)}-e^{-V(x)})}{|y-x|^{1+\alpha}}\,dy\,dx\\
 =&c_{1,\alpha}\lim_{\varepsilon \to0}\int_{\R}\int_{\{|y-x|\ge \varepsilon\}} \frac{(e^{-V(y)}-e^{-V(x)})}{|y-x|^{1+\alpha}}\,dy\,dx\\
 =&-c_{1,\alpha}\lim_{\varepsilon \to0}\int_{\R}\int_{\{|y-x|\ge \varepsilon\}} \frac{(e^{-V(x)}-e^{-V(y)})}{|y-x|^{1+\alpha}}\,dx\,dy\\
=& - c_{1,\alpha}\int_{\R}\lim_{\varepsilon \to0}\int_{\{|x-y|\ge \varepsilon\}} \frac{(e^{-V(x)}-e^{-V(y)})}{|x-y|^{1+\alpha}}
\,dx\,dy\\
=&\int_\R (-\Delta)^{\alpha/2}e^{-V(y)}\,dy, \end{align*} which proves \eqref{e:22aux1}.

   On the other hand,
 \begin{equation}\label{e:22aux2}
 \begin{split}
  -\int_{-\infty}^0(-\Delta)^{\alpha/2}e^{-V(u)}\,du&=\int_{-\infty}^0 \int_{\R}\left(e^{-V(u+z)}-e^{-V(u)}- \big[e^{-V(u)}\big]'z\I_{\{|z|\le 1\}}\right)\frac{c_{1,\alpha}}{|z|^{1+\alpha}}\,dz\,du\\
 &= \int_{-\infty}^0\lim_{\varepsilon \to 0}\int_{\{|z|\ge \varepsilon\} }\left(e^{-V( u+z)}-e^{-V( u)} \right)\frac{c_{1,\alpha}}{|z|^{1+\alpha}}\,dz\,du\\
 &=\int_0^\infty\lim_{\varepsilon \to 0}\int_{\{|z|\ge \varepsilon\} }\left(e^{-V(-u+z)}-e^{-V(-u)} \right)\frac{c_{1,\alpha}}{|z|^{1+\alpha}}\,dz\,du\\
 &=\int_0^\infty\lim_{\varepsilon \to 0}\int_{\{|z|\ge \varepsilon\} } \left(e^{-V(-u+z)}-e^{-V(u)} \right)\frac{c_{1,\alpha}}{|z|^{1+\alpha}}\,dz\,du\\
 &=\int_0^\infty\lim_{\varepsilon \to 0}\int_{\{|z|\ge \varepsilon\} }\left(e^{-V(-u-z)}-e^{-V(u)} \right)\frac{c_{1,\alpha}}{|z|^{1+\alpha}}\,dz\,du\\
  &=\int_0^\infty\lim_{\varepsilon \to 0}\int_{\{|z|\ge \varepsilon\} }\left(e^{-V( u+z)}-e^{-V(u)} \right)\frac{c_{1,\alpha}}{|z|^{1+\alpha}}\,dz\,du\\
  &=\int_0^\infty \int_{\R }\left(e^{-V( u+z)}-e^{-V(u)} - \big[e^{-V(u)}\big]'z\I_{\{|z|\le 1\}} \right)\frac{c_{1,\alpha}}{|z|^{1+\alpha}}\,dz\,du\\
  &=-\int^{\infty}_0(-\Delta)^{\alpha/2}e^{-V(u)}\,du,
  \end{split}
   \end{equation} where in the third and the fifth equalities we changed the variables, and the fourth and the sixth equalities follow from the
   symmetry
   $V(x)=V(-x)$ for all $x\in \R$.
 Combining \eqref{e:22aux1}
 with \eqref{e:22aux2},
  we have
  \begin{equation}\label{e:integralZero}
   \int^{\infty}_0(-\Delta)^{\alpha/2}e^{-V(u)}\,du=0
   \end{equation}
    and so $b(0)=0$.
 Furthermore, by $ \int^{\infty}_0(-\Delta)^{\alpha/2}e^{-V(u)}\,du=0$ and $(-\Delta)^{\alpha/2}e^{-V(x)}= (-\Delta)^{\alpha/2}e^{-V(-x)}$ for all $x\in \R$ (which is also due to the
 symmetry
 $V(x)=V(-x)$ for all $x\in \R$), we can get that for any $x<0$,
 $$b(x)=-e^{V(x)}\int_{-\infty}^x (-\Delta)^{\alpha/2} e^{-V(u)}\,du=-e^{V( x)}\displaystyle\int_{-x}^\infty (-\Delta)^{\alpha/2}e^{-V(z)}\,dz.$$
 The desired assertion \eqref{e:driftd1} follows.

As we mentioned
in the proof of Lemma \ref{Lemma:one1},
since $e^{-V}\in C_b^2(\R)$, $(-\Delta)^{\alpha/2}e^{-V(x)}\in C_b(\R)$.
 By  \eqref{e:integralZero}, we can easily see that
$b(x) \in C^1(\R)$ and is locally bounded.
\end{proof}

The following statement is analogous to Lemma \ref{Lemma3.1}.

\begin{lemma}\label{L:lemma5.1}
Let $V\in C^2(\R)$ be a symmetric function on $\R$
such that $e^{-V}\in L^1(\R;dx)\cap C_b^2(\R)$ and \eqref{cod-111} holds. Suppose that
\begin{equation}\label{cod1_1}
\lim_{x\rightarrow \infty}xe^{-V(x)}=0
\end{equation}
and
\begin{equation}\label{cod1_2}
 \liminf_{x\rightarrow \infty}[x^3e^{-V(x)}(V'(x)^2-V''(x))]\ge0.
\end{equation}
Then there exist constants $c_1,\ c_2>0$ and $r_1>1$ such that for all $x\in\R$,
$$
xb(x)\le c_1\I_{\{|x|\le r_1\}}-c_2\frac{e^{V(x)}}{|x| ^{1+\alpha}}|x|^2\I_{\{|x|>r_1\}}.
$$
\end{lemma}

\begin{proof} Since $b(x)$ is anti-symmetric,
we only need to consider $x\ge0$.
According to Lemma \ref{Lemma:second2},
$b(x) \in C^1(\R)$ and is
therefore
locally bounded. Hence, in order to prove the desired assertion, it is sufficient to verify that
there exists a constant $c>0$ such that for $x>0$ large enough
\begin{equation}\label{e:5.3}
-(-\Delta)^{\alpha/2}e^{-V(x)}\ge \frac{c}{x^{1+\alpha}}.
\end{equation}
To this end,  for $x>1$ we write
\begin{align*}
-(-\Delta)^{\alpha/2}e^{-V(x)}&=\int_{\R}\big(e^{-V(x+z)}-e^{-V(x)}+e^{-V(x)}V'(x)z\I_{\{|z|\le 1\}}\big)\frac{c_{1,\alpha}}{|z|^{1+\alpha}}\,dz\\
& =\int_{\{|z|<x/2\}}\big(e^{-V(x+z)}-e^{-V(x)}+e^{-V(x)}V'(x)z\big)\frac{c_{1,\alpha}}{|z|^{1+\alpha}}\,dz\\
&\quad +\int_{\{|z|\ge x/2\}}(e^{-V(x+z)}-e^{-V(x)})\frac{c_{1,\alpha}}{|z|^{1+\alpha}}\,dz\\
&=:I_1(x)+I_2(x).
\end{align*}

First, for $x>0$ large enough, we have
\begin{align*}
I_2(x)&=\int_{x/2}^\infty (e^{-V(x+z)}-e^{-V(x)})\frac{c_{1,\alpha}}{|z|^{1+\alpha}}\,dz+\int_{-x}^{-x/2}(e^{-V(x+z)}-e^{-V(x)})\frac{c_{1,\alpha}}{|z|^{1+\alpha}}\,dz\\
&\quad+\int_{-\infty}^{-x}(e^{-V(x+z)}-e^{-V(x)})\frac{c_{1,\alpha}}{|z|^{1+\alpha}}\,dz\\
&\ge \left(-e^{-V(x)}\int_{x/2}^\infty \frac{c_{1,\alpha}}{|z|^{1+\alpha}}\,dz\right)+ \left(\frac{c_{1,\alpha}}{x^{1+\alpha}}\int_0^{x/2}e^{-V(z)}\,dz-e^{-V(x)}\int_{x/2}^x \frac{c_{1,\alpha}}{|z|^{1+\alpha}}\,dz\right)\\
&\quad + \left(-e^{-V(x)}\int_{x}^\infty \frac{c_{1,\alpha}}{|z|^{1+\alpha}}\,dz\right)\\
&= \frac{c_{1,\alpha}}{x^{1+\alpha}}\int_0^{x/2}e^{-V(z)}\,dz -2e^{-V(x)}\int_{x/2}^\infty \frac{c_{1,\alpha}}{|z|^{1+\alpha}}\,dz\\
&\ge \frac{c_1}{x^{1+\alpha}}-c_2\frac{e^{-V(x)}}{x^\alpha}.
\end{align*}
On the other hand, by the Taylor theorem, for $x>0$ large enough,
\begin{align*}
I_1(x)&=\int_{\{|z|< x/2\}}(e^{-V(x+z)}-e^{-V(x)}+e^{-V(x)}V'(x)z)\frac{c_{1,\alpha}}{|z|^{1+\alpha}}\,dz\\
&\ge
c_{1,\alpha}\inf_{x/2\le z\le
3x/2}[e^{-V(z)}(V'(z)^2-V''(z))]\int_0^{x/2}z^{1-\alpha}\,dz\\
&=c_3 x^{2-\alpha}\inf_{x/2\le z\le 3x/2}[e^{-V(z)}(V'(z)^2-V''(z))]\\
&=c_3\frac{1}{x^{1+\alpha}} \Big[x^3\inf_{x/2\le z\le 3x/2}[e^{-V(z)}(V'(z)^2-V''(z))]\Big].
\end{align*}
Hence,  for $x>0$ large enough,
$$ -(-\Delta)^{\alpha/2}e^{-V(x)}\ge \frac{c_1}{x^{1+\alpha}}-\frac{c_2xe^{-V(x)}}{x^{1+\alpha}}
+c_3\frac{1}{x^{1+\alpha}} \Big[x^3\inf_{x/2\le z\le 3x/2}[e^{-V(z)}(V'(z)^2-V''(z))]\Big].$$  This along with \eqref{cod1_1} and \eqref{cod1_2} yields \eqref{e:5.3}. The proof is completed. \end{proof}

\begin{lemma}\label{Lemma5.2} Let $V\in C^2(\R)$ be a symmetric function on $\R$ such that
$e^{-V}\in L^1(\R;dx)\cap C_b^2(\R)$ and \eqref{cod-111} holds. If
\eqref{cod1_1} is satisfied,
then there exist constants $c_0>0$ and $r_1>1$ such that for all $x\in\R$ with $|x|\ge r_1$,
$$
b(x)\ge -c_0 \frac{e^{V(x)}}{|x| ^{ \alpha}}.
$$\end{lemma}
\begin{proof}The assertion follows from the conclusion that
there exists a constant $c_1>0$ such that for $x>0$ large enough
\begin{equation}\label{e:eff}
-(-\Delta)^{\alpha/2}e^{-V(x)}\le \frac{c_1}{x^{1+\alpha}}.
\end{equation} For \eqref{e:eff}, one can follow the idea for the argument of \eqref{e:5.3}. In particular,
under \eqref{cod-111}  it holds that
\begin{equation}\label{cod1_3}
 \limsup_{x\rightarrow \infty}[x^3e^{-V(x)}(V'(x)^2-V''(x))]<\infty.
\end{equation}
Then
we can deduce that
$$I_1(x)\le \frac{c_2}{x^{1+\alpha}}$$ by applying \eqref{cod1_3} instead of \eqref{cod1_2}. The details are omitted here. \end{proof}

\begin{remark} Under the assumptions of Lemma
\ref{L:lemma5.1},
for $|x|$ large enough,
$$xb(x) \asymp -\frac{e^{V(x)}}{(1+|x|)^{1+\alpha}}|x|^2.$$    \end{remark}

\section{
Properties of the SDE with the fractional drift}\label{sectionFurther}
In this section, we will consider the following stochastic differential
equation (SDE)
\begin{equation}\label{p2-1-00}dX_t=b(X_t)\,dt+dZ_t,\end{equation}
where $(Z_t)_{t\ge0}$ is a symmetric (rotationally invariant) $\alpha$-stable process on
$\R^d$ with $\alpha\in (0,2)$ and $d\ge1$, and $b(x)$ is defined by \eqref{specialDrift} when $d>2-\alpha$ and by \eqref{e:one} when $d\le 2-\alpha$.
Everywhere below, we
assume that Assumption {\rm(A)} is satisfied.

Suppose
first
that $d>2-\alpha$. According to Lemmas \ref{L:lemma0} and \ref{Lemma3.1}, for the drift $b(x)$ defined by \eqref{specialDrift}, we have
$b\in C^\beta(\R^d)$ with
$\beta=2-\alpha$ when $\alpha\in (1,2)$, $\beta=1-\varepsilon$ for any $\varepsilon>0$ when $\alpha=1$, and $\beta=1$ when $\alpha\in (0,1)$ (in particular, $b\in C^\beta(\R^d)$ with $\beta\in (0,1-\alpha/2)$
for all $\alpha\in (0,2)$),
 and
\begin{equation}\label{p2-1-0}\langle b(x),x\rangle\le K(1+|x|^2),\quad x\in\R^d\end{equation} for some constant
$K>0$, where $C^\beta(\R^d)$ denotes the set of locally $\beta$-H\"{o}lder continuous functions from $\R^d$ to $\R^d$ for $\beta\in (0,1)$. Suppose
now
that $d\le 2-\alpha$. Then, by Lemmas \ref{Lemma:second2}
and \ref{L:lemma5.1},
the drift $b(x)$ defined by  \eqref{e:one} belongs to $C^1(\R)$ and satisfies \eqref{p2-1-0} as well.
Here we used the fact that \eqref{cod1_1}
holds
 under condition \eqref{e:1.1}
and hence under Assumption {\rm(A)}, all the conditions required in Lemmas \ref{Lemma:second2}
and \ref{L:lemma5.1} are satisfied.
 Consequently, for all $d \ge 1$ and $\alpha \in (0,2)$,
  the equation
\eqref{p2-1-00} has a unique non-explosive strong solution
$(X_t)_{t\ge0}$, which is a strong Markov process with the generator
$$Lf(x)=-(-\Delta)^{\alpha/2}f(x)+\langle b(x),\nabla f(x)\rangle,\quad f\in C_b^2(\R^d).$$
For the case of $d>2-\alpha$, the reader can be referred to  \cite[Theorem 2.4 and Lemma 7.1]{XZ},
while for $d \le 2-\alpha$ one can directly apply e.g.\ \cite[Theorem 1.1]{Majka2016}, since $b \in C^1(\R)$ obviously implies that $b(x)$ satisfies a local Lipschitz condition.
Alternatively,
for any $d \ge 1$ and $\alpha \in (0,2)$,
we can first apply \cite[Theorem 1.1]{Pr} or \cite[Corollary 1.4(i)]{CSZ} (with $b\in C_b^\beta(\R^d)$, i.e., with $b(x)$ being globally $\beta$-H\"{o}lder continuous) to get the locally unique strong solution, and then use the additional global one-sided linear growth
condition \eqref{p2-1-0} to obtain the unique non-explosive strong solution; see the proof of \cite[Theorem 1]{GyongyKrylov} or \cite[Theorem 1.1]{Majka2016}.

In the following, we will prove rigorously that \eqref{Gibbs} is indeed
 the unique invariant measure for the process $(X_t)_{t \ge 0}$ defined as the solution to \eqref{stableSDE} with the drift term $b(x)$ defined by \eqref{specialDrift} and \eqref{e:one}.

We begin with the following simple lemma.

\begin{lemma}\label{L:4.1}Under Assumption {\rm(A)}, for any $\beta\in (0,\alpha)$, there are constants $C_1,C_2>0$ such that for all $x\in \R^d$,
\begin{equation}\label{e:cougen1}L V_0(x)\le C_1- C_2 \frac{e^{V(x)}}{|x|^{d+\alpha}} V_0(x), \end{equation}where $V_0(x)=(1+|x|^2)^{\beta/2}$. \end{lemma}
\begin{proof} According to Lemmas \ref{Lemma3.1} and \ref{L:lemma5.1}, we know that under Assumption (A) there are constants $\lambda_1, \lambda_2>0$ such that
for all $x\in \R^d$,
\begin{equation}\label{e:123}\langle x, b(x)\rangle\le \lambda_1-\lambda_2 U(x)|x|^2, \end{equation} where $U(x)=e^{V(x)}/(1+|x|)^{d+\alpha}.$
Here, we used again the fact that \eqref{cod1_1} holds true under condition \eqref{e:1.1} .

Recall that $\frac{c_{d,\alpha}}{|z|^{d+\alpha}}$
is the density function of the L\'evy measure for the symmetric $\alpha$-stable process.
Since $\nabla V_0(x)=\beta(1+|x|^2)^{(\beta-2)/2}x$ and
$\| \nabla^2 V_0 \|_{\infty} \le
\beta\left( 2-\beta/2 \right)$,
we find that for all $x\in \R^d$ and $l\ge1$,
\begin{align*}LV_0(x)=&\beta(1+|x|^2)^{(\beta-2)/2} \langle x,b(x)\rangle+\int_{\{|z|\le l\}} (V_0(x+z)-V_0(x)-\langle\nabla V_0(x),z\rangle)\frac{c_{d,\alpha}}{|z|^{d+\alpha}}\,dz\\
&+\int_{\{|z|>l\}} (V_0(x+z)-V_0(x))\frac{c_{d,\alpha}}{|z|^{d+\alpha}}\,dz\\
\le&\beta(1+|x|^2)^{(\beta-2)/2} \langle x,b(x)\rangle+
(\beta/2)\left(2 - \beta/2 \right) c_{d,\alpha}
 \int_{\{|z|\le l\}}\frac{1}{|z|^{d+\alpha-2}}\,dz\\
&+\int_{\{|z|>l\}} [(1+2|x|^2)^{\beta/2}+(2|z|^2)^{\beta/2}]\frac{c_{d,\alpha}}{|z|^{d+\alpha}}\,dz\\
\le &\beta(1+|x|^2)^{(\beta-2)/2} (\lambda_1-\lambda_2U(x)|x|^2)+c_1 l^{2-\alpha}+c_2 l^{-\alpha} (1+2|x|^2)^{\beta/2}+c_3, \end{align*}where $c_i$ $(1\le i\le 3)$ are independent of $l$ and $x\in \R^d$.
Here, in the equality above, we used the fact that $\int_{\{1\le |z|\le l\}} z\frac{c_{d,\alpha}}{|z|^{d+\alpha}}\,dz=0$; the first inequality follows from the mean value theorem and the fact that $V_0(x+z)\le (1+2|x|^2+2|z|^2)^{\beta/2}\le (1+2|x|^2)^{\beta/2}+(2|z|^2)^{\beta/2}$; and in the last inequality we used  \eqref{e:123} and the facts that
$$\int_{\{|z|\le l\}}\frac{1}{|z|^{d+\alpha-2}}\,dz\le c_4l^{2-\alpha},\quad \int_{\{|z|>l\}}  \frac{1}{|z|^{d+\alpha}}\,dz\le c_4 l^{-\alpha}$$ and
$$\int_{\{|z|\ge 1\}}\frac{1}{|z|^{d+\alpha-\beta}}\,dz<\infty,\quad \beta\in [0,\alpha).$$
From the right hand side of the inequality above, we can see that $LV_0(x)$ is locally bounded, and for $|x|$ large enough,
$$LV_0(x)\le -\frac{\lambda_2\beta}{2} U(x){|x|^{\beta}}+ c_1l^{2-\alpha}+ 4c_2 l^{-\alpha}|x|^{\beta},$$ which is dominated by $-\frac{\lambda_2\beta}{4} U(x){|x|^{\beta}}$ by choosing $|x|\gg l\gg 1$. Then, \eqref{e:cougen1} follows.\end{proof}

We also need the following statement.

\begin{lemma}\label{sde} Let $(X_t)_{t\ge0}$ be the
unique strong solution to the SDE \eqref{p2-1-00}  with $b(x)$ defined by \eqref{specialDrift} when $d>2-\alpha$ and by \eqref{e:one} when $d\le 2-\alpha$, such that Assumption {\rm(A)} is satisfied.  Then,
\begin{itemize}
\item[{\rm(i)}] The process $(X_t)_{t\ge0}$ is strong Feller and Lebesgue irreducible;
\item[{\rm(ii)}]The transition probability function of the process $(X_t)_{t\ge0}$ is absolutely continuous with respect to the Lebesgue measure.
\end{itemize}
In particular, the process has a unique invariant probability measure $\mu(dx)=\rho(x)\,dx$, where $\rho(x)>0$ for all $x\in \R^d$. \end{lemma}

 \begin{proof}  For simplicity, we only consider the case of $d>2-\alpha$, since the case of $d\le 2-\alpha$ can be proved similarly and easily.

 {\rm(i)} For any $n\ge1$, let $$b_n(x) = -C_{d,2-\alpha} e^{V(x)\wedge K(n)} \int_{\R^d} \frac{e^{-V(y)} \nabla V(y)}{|x-y|^{d-(2-\alpha)}}\, dy,$$ where
 $$K(n)=1+\sup_{|x|\le n}|V(x)|.$$
 Then, according to the proof of Lemma \ref{L:lemma0}, the function $x\mapsto\int_{\R^d} \frac{e^{-V(y)} \nabla V(y)}{|x-y|^{d-(2-\alpha)}}\, dy $ is bounded, and globally $(2-\alpha)$-H\"{o}lder continuous when $\alpha\in (1,2)$,
  globally $(1-\varepsilon)$-H\"{o}lder continuous for any $\varepsilon>0$ when $\alpha=1$, and belongs to $C^1(\R^d)$ when $\alpha\in (0,1)$,
  and
  hence $b_n(x)$ also shares these properties.
  Consider the following SDE
\begin{equation}\label{p2-1-1}
dX_t^{(n)}=b_n(X_t^{(n)})\,dt+dZ_t.
\end{equation}
It follows from \cite[Theorem 1.1]{Pr} or \cite[Corollary 1.4(i)]{CSZ} that the SDE \eqref{p2-1-1} has a unique strong solution, which will be denoted by $X^{(n)}:=(X^{(n)}_t)_{t\ge0}$.
Note that the infinitesimal generator of the process $X^{(n)}$ is given by
$$L^{(n)}f(x)=\langle b_n(x),\nabla f(x)\rangle-(-\Delta)^{\alpha/2}f(x),\quad f\in C_b^2(\R^d).$$
Hence, according to \cite[Theorem 1.5]{CZ} for $\alpha\in (1,2)$ and \cite[Theorem 1.1]{XZ1} for $\alpha=1$ as well as \cite[Theorem 2.2]{Kul} for $\alpha\in (0,1)$, the process $X^{(n)}$ has a  continuous and strictly positive transition density function, which implies that $X^{(n)}$  is strong Feller (i.e., for any $f\in B_b(\R^d)$ and $t>0$, $x\mapsto P^{(n)}_t f(x):=\Ee^x f(X_t^{(n)})$ is continuous)
 and Lebesgue irreducible (i.e., for any $t>0$ and open set $O\in \mathscr{B}(\R^d)$ with $\leb(O)>0$, $\Pp^x(X_t^{(n)}\in O)>0$). Here and in what follows, we assume that $X$ and $X^{(n)}$ are defined on the same probability space $(\Omega, \mathscr{F},\Pp)$. Let $\Pp^x(\cdot)=\Pp(\cdot|X_0=x)$ or $\Pp^x(\cdot)=\Pp(\cdot|X^{(n)}_0=x)$ without confusion.
Since
$b_n(x)=b(x)$ for all $|x|\le n$,
the law of $X_{t \wedge \tau_n}$ is the same as the law of $X_{t \wedge \tau_n}^{(n)}$ for any $t > 0$,
where $\tau_n:=\inf\{t>0:|X_t|\ge n\}$.

Now, let $(P_t)_{t\ge0}$ be the semigroup of the process $X$. For any $f\in B_b(\R^d)$, $x_0\in\R^d$ and for any sequence $\{x_k\}_{k\ge1}\subseteq \R^d$ such that $x_k \rightarrow x_0$ as $k \rightarrow \infty$, we choose $n$ large enough so that $\{x_k\}_{k\ge0}\subset B(0,n)$, and then find that
\begin{equation}\label{e:conv1}\begin{split}
 &|P_t f(x_k)-P_t f(x_0)|\\
 &=|\Ee^{x_{k}}f(X_t)-\Ee^{x_0}f(X_t)|\\
 &\le |\Ee^{x_{k}}(f(X_t)\I_{\{t<\tau_n\}})-\Ee^{x_0}(f(X_t)\I_{\{t<\tau_n\}})|+ \|f\|_{\infty}\big(\Pp^{x_k}(\tau_n\le t)+\Pp^{x_0}(\tau_n\le t)\big)\\
 &=   |\Ee^{x_{k}}(f(X^{(n)}_t)\I_{\{t<\tau_n\}})-\Ee^{x_0}(f(X^{(n)}_t)\I_{\{t<\tau_n\}})|+ \|f\|_{\infty}\big(\Pp^{x_k}(\tau_n\le t)+\Pp^{x_0}(\tau_n\le t)\big)\\
&\le |\Ee^{x_{k}}f(X^{(n)}_t)-\Ee^{x_0}f(X^{(n)}_t)|
+2\|f\|_{\infty}\big(\Pp^{x_k}(\tau_n\le t)+\Pp^{x_0}(\tau_n\le t)\big)\\
&\le |P_t^{(n)}f(x_k)-P_t^{(n)} f(x_0)|+4\|f\|_{\infty}\sup_{k\ge0}\Pp^{x_k}(\tau_n\le t).
\end{split}\end{equation}

Note that, combining Lemma \ref{L:4.1} with the standard argument (for example, see the proof of \cite[Theorem 2.1]{MT}), we can see that for any $k\ge0$ and $t>0$,
$$\Pp^{x_k}(\tau_n\le t)=\Pp^{x_k}(\max_{s\in [0,t]}|X_s|\ge n)=\Pp^{x_k}\left( \max_{s\in[0,t]}(1+|X_s|^2)^{\beta/2}\ge (1+|n|^2)^{\beta/2}\right) \le\frac{ c_1(1+|x_k|^2)^{\beta/2}}{(1+|n|^2)^{\beta/2}}.$$
Since $x_k \rightarrow x_0$ as $k \rightarrow \infty$, without loss of generality we may and will assume that $x_k\in B(x_0,1)$. Hence,
$$\lim_{n\to\infty}\sup_{k\ge0}\Pp^{x_k}(\tau_n\le t)=0.$$

Letting $k \rightarrow \infty$ and then $n \rightarrow \infty$ in \eqref{e:conv1}, we show that $$\lim_{k\to \infty}|P_t f(x_k)-P_t f(x_0)|=0.$$ Hence, for any $f\in B_b(\R^d)$ and $t>0$, $P_tf$ is a continuous function, i.e., the process $X$ is strong Feller.

For any $x\in \R^d$, $t>0$ and open set $O\in \mathscr{B}(\R^d)$ with $\leb(O)>0$, choosing $n$ large enough such that $\leb(O\cap B(0,n))>0$,
\begin{align*} \Pp^x(X_t\in O)\ge &\Pp^x(X_t\in O, \tau_n >t)=\Pp^x(X^{(n)}_t\in O\cap B(0,n), \tau_n >t).\end{align*}
According to (the proof of) \cite[Corollary 3.6]{CKS},
the Dirichlet heat kernel of the process $X^{(n)}$ is positive everywhere, and so the right hand side of the inequality above is positive (even though the setting of \cite{CKS} is restricted to $d\ge2$, the proof of \cite[Corollary 3.6]{CKS} is based on the global heat kernel estimates and the L\'evy system for $X^{(n)}$, both of which are available for $d=1$ too, and so \cite[Corollary 3.6]{CKS} holds true for all $d\ge1$).
Hence,
$\Pp^x(X_t\in O)>0$ and thus the process $X$ is Lebesgue irreducible.

Therefore, all compact sets are petite for $X$ (cf.\ \cite[Theorem 4.1(i)]{MT2}), and hence the existence of the invariant probability measure $\mu$ follows from \eqref{e:cougen1}, while the uniqueness is a direct consequence of the strong Feller property and irreducibility; see \cite[Theorems 5.1 and 5.2]{MT}.

{\rm(ii)} As we already established in the first part of the proof, according to \cite[Theorem 1.5]{CZ}, for any $t>0$, the law of $X_t^{(n)}$ is absolutely continuous with respect to the Lebesgue measure. We will claim that
the law of $X_t$ is also absolutely continuous with respect to the Lebesgue measure. Indeed, for any open set $O\in \mathscr{B}(\R^d)$ such that
$\leb(O)=0$, any $t>0$, $x \in \R^d$ and $n$ large enough,
\begin{align*} \Pp^x(X_t\in O)=&\Pp^x(X_t\in O, \tau_n >t)+\Pp^x(X_t\in O, \tau_n \le t) \\
= &\Pp^x(X^{(n)}_t\in O, \tau_n >t)+\Pp^x(X_t\in O, \tau_n \le t)\\
\le &\Pp^x(X^{(n)}_t \in O)+2\Pp^x( \tau_n \le t)= 2\Pp^x( \tau_n \le t).
\end{align*} As mentioned above, for any $x\in \R^d$ and $t>0$, $\Pp^x( \tau_n \le t)\to0$ as $n\to \infty$. Hence, $\Pp^x(X_t \in
O)=0$ for any $x\in \R^d$ and $t>0$.

Let $P(t,x,\cdot)$ be the transition function of the process $X$.
By the argument for the Lebesgue irreducibility above, we
know that $P(t,x,\cdot)$ and the Lebesgue measure are equivalent, so that
$P(t,x,A)=\int_A p(t,x,y)\,dy$ for any $A\in \mathscr{B}(\R^d)$ and $p(t,x,y)$ can be chosen to be strictly positive everywhere on $\R^d\times \R^d$ for any fixed $t>0$. Hence, for the invariant probability measure $\mu$,
since $\mu(A)=\int_{\R^d} P(t,x,A)\,d\mu(x)$ for $A\in \mathscr{B}(\R^d)$ and $t>0$, $\mu$ is also absolutely continuous with respect to the Lebesgue measure and the associated density function can be chosen to be strictly positive everywhere.
\end{proof}

\begin{proposition}\label{prop:inv} Let $X:=(X_t)_{t\ge0}$ be the
unique strong solution to the SDE \eqref{stableSDE}  with $b(x)$ defined by \eqref{specialDrift} when $d>2-\alpha$ and by \eqref{e:one} when $d\le 2-\alpha$ such that Assumption {\rm(A)} is satisfied.  Then, $\mu(dx):=Z^{-1} e^{-V(x)}\,dx$ with $Z=\int_{\R^d} e^{-V(x)}\,dx$ is the unique invariant probability measure for the process $X$.  \end{proposition}
\begin{proof}

Recall that the infinitesimal generator of the process $(X_t)_{t\ge0}$ is given by
$$Lf(x)=-(-\Delta)^{\alpha/2} f(x)+\langle b(x), \nabla f(x)\rangle.$$
Let $D(L)$ be the domain of the operator under the norm $\|\cdot\|_\infty$. Then,
if $\mu$ is an invariant measure for $(P_t)_{t \ge 0}$,
for any $f\in D(L)$,
\begin{equation}\label{e:invari}\mu(Lf)=\mu\left(\lim_{t\to0}\frac{ P_tf-f}{t}\right)=\lim_{t\to 0} \frac{\mu(P_tf)-\mu(f)}{t}=0.\end{equation} Actually, \eqref{e:invari} is equivalent to saying that $\mu$ is an invariant probability measure of the process $X$
and this is still true if we replace $D(L)$ with a core;
see e.g.\ \cite[Theorem 3.37]{Li}.

According to \cite[Theorem 1.5]{CZ}, $C_b^2(\R^d)$ is contained in the domain of the infinitesimal generator of the process $X^{(n)}$ given by the SDE \eqref{p2-1-1}.  Then, by the localization argument that we used in the proof of the strong Feller property above, we can check that $C_c^\infty(\R^d)\subset D(L)$.
In the following, we take $\mu(dx):=Z^{-1} e^{-V(x)}\,dx$ with $Z=\int_{\R^d} e^{-V(x)}\,dx$, and verify that for any $f\in C_c^\infty(\R^d)$,
$\mu(Lf)=0$.

Let us first suppose that $d>2-\alpha$. Then, for $b(x)$ defined by \eqref{specialDrift}  and for any $f\in C_c^\infty(\R^d)$,
\begin{equation}\label{e:kkss}\begin{split}\int_{\R^d} Lf(x) e^{-V(x)}\,dx&=-\int_{\R^d} e^{-V(x)}(-\Delta)^{\alpha/2} f(x)\,dx+\int_{\R^d}e^{-V(x)}\langle\nabla f(x),b(x)\rangle\,dx\\
&=-\int_{\R^d} e^{-V(x)}(-\Delta)^{\alpha/2} f(x)\,dx\\
&\quad+C_{d,2-\alpha}\int_{\R^d}\left\langle\nabla f(x),\nabla \left[\int_{\R^d} \frac{ e^{-V(y)}}{|\cdot-y|^{d-(2-\alpha)}}\,dy\right](x)\right\rangle\,dx. \end{split}\end{equation} On the other hand, by the integration by parts, we find that for any $f\in C_c^\infty(\R^d)$,
\begin{align*}&C_{d,2-\alpha}\int_{\R^d}\left\langle\nabla f(x),\nabla \left[\int_{\R^d} \frac{ e^{-V(y)}}{|\cdot-y|^{d-(2-\alpha)}}\,dy\right](x)\right\rangle\,dx\\
&=C_{d,2-\alpha}\int_{\R^d} (-\Delta) f(x)   \int_{\R^d} \frac{ e^{-V(y)}}{|x-y|^{d-(2-\alpha)}}\,dy \,dx\\
&=C_{d,2-\alpha}\int_{\R^d} (-\Delta)^{1-\alpha/2} \big[(-\Delta)^{\alpha/2}f\big](x)  \int_{\R^d} \frac{ e^{-V(y)}}{|x-y|^{d-(2-\alpha)}}\,dy \,dx\\
&=C_{d,2-\alpha}\int_{\R^d} (-\Delta)^{\alpha/2}f (x) \cdot (-\Delta)^{1-\alpha/2} \left[\int_{\R^d} \frac{ e^{-V(y)}}{|\cdot-y|^{d-(2-\alpha)}}\,dy\right](x) \,dx\\
&= \int_{\R^d} e^{-V(x)} (-\Delta)^{\alpha/2}f (x) \,dx,\end{align*} where in the second equality we used the fact that $(-\Delta)= (-\Delta)^{\alpha/2} (-\Delta)^{1-\alpha/2}$ (which can be checked by the standard Fourier analysis),
the third equality follows from the symmetry of $(-\Delta)^{1-\alpha/2}$ on $L^2(\R^d;dx)$, and in the fourth equality we used the fact that $\frac{C_{d,2-\alpha}}{|x-y|^{d-(2-\alpha)}}$ is the Green function for the symmetric $(2-\alpha)$-stable process, and hence for all $x\in \R^d$,
$$  (-\Delta)^{1-\alpha/2}\left[\int_{\R^d} \frac{ C_{d,2-\alpha}e^{-V(y)}}{|\cdot-y|^{d-(2-\alpha)}}\,dy\right](x)=e^{-V(x)},$$
cf.\ \cite[Proposition 7.2]{Kwa}.
The equality above along with \eqref{e:kkss} yields that $\int_{\R^d} Lf(x) e^{-V(x)}\,dx=0$, and so the desired assertion follows.

Now, we consider the case that $d\le 2-\alpha$; i.e., $d=1$ and $\alpha\in (0,1]$. For $b(x)$ defined by \eqref{e:one}, using \eqref{e:driftd1}, we have for any $f\in C_c^\infty(\R)$,
\begin{align*}&\int_{\R} Lf(x) e^{-V(x)}\,dx\\
&=-\int_{\R} e^{-V(x)}(-\Delta)^{\alpha/2} f(x)\,dx+\int_{0}^\infty f'(x)\int_x^\infty (-\Delta)^{\alpha/2}e^{-V(z)}\,dz \,dx\\
&\quad-\int_{-\infty}^0 f'(x) \int_{-x}^\infty (-\Delta)^{\alpha/2}e^{-V(z)}\,dz\,dx\\
&=-\int_{\R} e^{-V(x)}(-\Delta)^{\alpha/2} f(x)\,dx+\int_{0}^\infty f(x)(-\Delta)^{\alpha/2}e^{-V(x)}\,dx+\int_{-\infty}^0 f(x)  (-\Delta)^{\alpha/2}e^{-V(-x)}\,dx.\\
&=-\int_{\R} e^{-V(x)}(-\Delta)^{\alpha/2} f(x)\,dx+ \int_{\R} e^{-V(x)}(-\Delta)^{\alpha/2} f(x)\,dx=0, \end{align*} where in the second equality we used the fact that $\int_0^\infty  (-\Delta)^{\alpha/2}e^{-V(z)}\,dz=0$ (cf.\
\eqref{e:integralZero}) and the third equality follows from the fact that $ (-\Delta)^{\alpha/2}e^{-V(-x)}= (-\Delta)^{\alpha/2}e^{-V(x)}$ for all $x\le 0$ due to the symmetry of $V(x)$.

Therefore, according to both conclusions above and Lemma \ref{sde}, we prove that $\mu(dx):=Z^{-1} e^{-V(x)}\,dx$ is the unique invariant probability measure of the process $X$.
\end{proof}

\begin{remark}\label{e:remk--}
When $d>2-\alpha$, by some elementary calculations, the dual of the operator $L$ on $L^2(\R^d;dx)$ is given by
\begin{align*}L^*f(x)= & -(-\Delta)^{\alpha/2} f(x)-\langle b(x), \nabla f(x)\rangle-\D b(x) f(x)\\
= & -(-\Delta)^{\alpha/2} f(x)-\D (b f)(x).\end{align*}
Arguing informally,
we have \begin{align*} \D (b e^{-V})(x)&={\rm div}[\nabla (-\Delta)^{-(1-\alpha/2)}e^{-V}](x)
  =[\Delta(-\Delta)^{-(1-\alpha/2)}e^{-V}](x)\\
  =&-(-\Delta)(-\Delta)^{-(1-\alpha/2)}e^{-V}(x) =-(-\Delta)^{\alpha/2}e^{-V}(x),\end{align*} and so, by \eqref{exa-b},
$L^* e^{-V}(x)=0$ for $x\in \R^d$,
which would imply the infinitesimal invariance of $\mu$ given by \eqref{Gibbs} for the process $(X_t)_{t \ge 0}$ defined by \eqref{stableSDE},
cf.\ the proof of \cite[Theorem 1.1]{Simsekli2017}.
 However, since we do not know whether $\nabla (-\Delta)^{-(1-\alpha/2)}e^{-V}$ belongs to $C^1(\R^d)$ or not when $\alpha\in (1,2)$ (cf.\ Remark \ref{remark:notC1}), $\D\nabla (-\Delta)^{-(1-\alpha/2)}e^{-V}$ may be not well defined. Hence the argument above is informal and, in order to rigorously prove that $\mu$ is the unique invariant measure of $(X_t)_{t \ge 0}$, it is necessary to argue as in the proof of Proposition \ref{prop:inv}.
\end{remark}

Using Lemma \ref{sde} and Proposition \ref{prop:inv}, we can now easily prove Theorem \ref{mainResult}.

\begin{proof}[Proof of Theorem $\ref{mainResult}$]
	From Lemma \ref{sde}, we know that the process $X := (X_t)_{t \ge 0}$ obtained as the unique solution to the SDE \eqref{p2-1-00} is strong Feller and irreducible. Hence, due to \cite[Theorem 4.1(i)]{MT2},
all compact sets are petite for $X$. Moreover,
according to Lemma \ref{L:4.1}, we have the Lyapunov condition \eqref{e:cougen1}. As a consequence, \cite[Theorem 6.1]{MT} applies, and so
	there is a constant $\lambda > 0$
	such that for any $x\in \R^d$ and $t>0$,
$$\|P(t,x,\cdot)-\mu\|_{{\rm Var},V_0}\le C(x) V_0(x) e^{-\lambda t},$$
	where $V_0(x)=(1+|x|^2)^{\beta/2}$ with $\beta\in (0,\alpha)$,
$C(x)$ is a non-negative and locally bounded function on $\R^d$,
 and $\mu$ is the unique invariant probability measure for $X$. Finally, from Proposition \ref{prop:inv} we know that $\mu$ is given by \eqref{Gibbs}, and the proof is concluded.
\end{proof}

\bigskip

\noindent{\bf Acknowledgement.}
Mateusz B.\ Majka would like to thank Aleksandar Mijatovi\'{c} for discussions regarding Fractional Langevin Monte Carlo, and Jian Wang would like to thank Professor Renming Song and Dr.\ Longjie Xie for helpful comments on heat kernel estimates for SDEs with L\'evy jumps.
The research of Lu-Jing Huang is supported by the National Natural Science Foundation of China (No.\ 11901096).
 A part of this work was completed while Mateusz B.\ Majka was affiliated to the University of Warwick and supported by the EPSRC grant no.\ EP/P003818/1. The research of Jian Wang is supported by the National Natural Science Foundation of China (No.\ 11831014),
the Program for Probability and Statistics:
Theory and Application (No.\ IRTL1704),
and the Program for Innovative Research Team in Science and Technology
in Fujian Province University (IRTSTFJ).

\end{document}